\documentclass[reqno]{amsart}
\usepackage{amssymb, amsfonts}
\usepackage{enumerate}
\usepackage[usenames, dvipsnames]{color}
\usepackage{graphicx}

\usepackage{verbatim}
\usepackage{hyperref}

\numberwithin{equation}{section}

\newtheorem{theorem}{Theorem}[section]

\newtheorem{lemma}[theorem]{Lemma}
\newtheorem{proposition}[theorem]{Proposition}
\newtheorem*{AssumpA}{Assumption A$(\rho,\varepsilon)$}
\newtheorem*{AssumpA'}{Assumption A$'$ $(\rho,\varepsilon)$}
\newtheorem*{AssumpA''}{Assumption A$''$ $(\rho,\varepsilon)$}

\theoremstyle{definition}
\newtheorem{remark}[theorem]{Remark}

\theoremstyle{definition}

\theoremstyle{definition}

\makeatletter
\def\dashint{\operatorname%
{\,\,\text{\bf--}\kern-.98em\DOTSI\intop\ilimits@\!\!}}
\makeatother

\def\sfx{{\sf x}}

\def\bC{\mathbb{C}}

\def\bL{\mathbb{L}}

\def\bR{\mathbb{R}}
\def\bW{\mathbb{W}}
\def\bZ{\mathbb{Z}}

\def\fH{\mathfrak{H}}

\def\cD{\mathcal{D}}

\def\cF{\mathcal{F}}

\def\cM{\mathcal{M}}
\def\cO{\mathcal{O}}

\def\cQ{\mathcal{Q}}

\def\Xint#1{\mathchoice
{\XXint\displaystyle\textstyle{#1}}%
{\XXint\textstyle\scriptstyle{#1}}%
{\XXint\scriptstyle\scriptscriptstyle{#1}}%
{\XXint\scriptscriptstyle\scriptscriptstyle{#1}}%
\!\int}
\def\XXint#1#2#3{{\setbox0=\hbox{$#1{#2#3}{\int}$ }
\vcenter{\hbox{$#2#3$ }}\kern-.6\wd0}}

\def\dashint{\Xint-}

\begin{document}

\title[Parabolic Systems]{Parabolic Systems with measurable coefficients in weighted Sobolev spaces}

\author{Doyoon Kim}
\thanks{The first author was supported by the National Research Foundation of Korea (NRF) grant funded by the Korea government (MSIT) (2019R1A2C1084683).}
\address{Doyoon Kim, Department of Mathematics, Korea University, Anam-ro 145, Sungbuk-gu, Seoul, 02841, Republic of Korea}
\email{doyoon\_kim@korea.ac.kr}

\author{Kyeong-Hun Kim}
\thanks{The second author was supported by the National Research Foundation of Korea (NRF) grant funded by the Korea government (MSIT) (2019R1A5A1028324)}
\address{Kyeong-Hun Kim, Department of Mathematics, Korea University, Anam-ro 145, Sungbuk-gu, Seoul, 02841, Republic of Korea}
\email{kyeonghun@korea.ac.kr}

\author{Kijung Lee}
\thanks{The third author was supported by the National Research Foundation of Korea (NRF) grant funded by the Korea government (MSIT) (2019R1F1A1058988)}
\address{Kijung Lee, Department of Mathematics, Ajou University, World cup-ro 206, Yeongtong-gu, Suwon, 16499, Republic of Korea}
\email{kijung@ajou.ac.kr}


\begin{abstract}
We present a weighted $L_p$-theory of  parabolic systems on a half space $\bR^d_+$.  The leading coefficients are assumed to be only measurable in time $t$  and  have small bounded mean oscillations (BMO)  with respect to the spatial variables $x$, and the lower order coefficients are allowed to blow up near the boundary.
\end{abstract}

\maketitle

\section{Introduction}
\label{sec:Introduction}
In this paper we propose a weighted $L_p$-theory  for parabolic systems in the following non-divergence form:
\begin{equation}\label{eq_intro_2}
-u_t(t,x) + \sum_{i,j=1}^dA^{ij}(t,x) D_{ij} u(t,x) +\sum_{i=1}^dB^i(t,x)D_i u+C(t,x)u-\lambda u(t,x) = f(t,x)
\end{equation}
defined on $(-\infty, T)\times \bR^d_+$, where $\bR^d_+:= \{x=(x_1,x')\in \bR^d: x_1>0\}$ and $\lambda$ is a non-negative constant.  The coefficient matrices $A^{ij} = [a^{ij}_{kr}]$ for $i,j=1,\ldots,d$,  $B^{i}=[b^i_{kr}]$ for $i=1,\ldots,d$,  and  $C=[c_{kr}]$ have dimension  $d_1\times d_1$ and depend on $(t,x)$. The free term $f$ and the solution $u$ are $d_1\times 1$ matrix valued functions, that is
$$
u=
\begin{bmatrix}
u^1\\  \vdots \\ u^{d_1}\\
\end{bmatrix},
\quad
f = 
\begin{bmatrix}
f^1\\  \vdots \\ f^{d_1}\\
\end{bmatrix},
$$
where the entries can take values  in  $\bC$.

We may interpret $u(t,\cdot)$ as a family of  densities of diffusing chemical materials in a medium at time $t$. 
The system \eqref{eq_intro_2} combined with the zero boundary condition, the typical control of the densities on the boundary, yields a very subtle question on the behavior of solutions near the boundary, for instance, when the free term $f$ blows up near the boundary, since the densities are forced to decrease or increase near the boundary in a very steep way. 
Thus, the behavior of solutions near the boundary is quite different from that in the interior of the domain and can be measured via  weights that involve the distance to the boundary, which is $x_1$  in this paper.
We aim to understand  the  regularity relation between the solution $u$ and the free term $f$,  mainly focusing on the boundary behaviors of the solution.   The leading coefficients are assumed to be only measurable in time $t$  and  have small bounded mean oscillations (BMO)  with respect to the spatial variables $x$.

We consider  the system (\ref{eq_intro_2}) in the framework of weighted Sobolev spaces  
\begin{equation*}
L_p((-\infty,T); H^{\gamma}_{p,\theta}(\bR^d_+)),\label{space_framework}
\end{equation*} 
which were introduced  by  N. V. Krylov  \cite{MR1708104} with $\gamma\in \bR$. 
If $\gamma$ is a non-negative integer, we have  the  characterization
\begin{equation*}
							\label{eq0615_01}
H_{p, \theta}^{\gamma}:=H_{p, \theta}^{\gamma}(\bR^d_+)
= \{ u : x_1^{|\alpha|} D^{\alpha} u \in L_{p, \theta}(\bR^{d}_{+}) \;\;
\forall \alpha : 0 \le |\alpha| \le \gamma \},
\end{equation*}
where $L_{p, \theta}(\bR^{d}_{+})$ is the $L_p$-space with the weighted Lebesgue  measure
$\mu_d(dx) = x_1^{\theta - d} \, dx$.  Since the work of \cite{MR1708104}, there has been steady attentions to the solvability theory for equations in the weighted Sobolev spaces $H_{p,\theta}^\gamma$ setting; see \cite{MR2111792, MR2561181, MR2990037, MR3147235}. The necessity of such weighted Sobolev spaces comes from, for instance, the theory of stochastic partial differential equations (SPDEs);  see  e.g. \cite{MR1262972, MR1720129} for detailed motivations. In short, we point out that, in general,  the derivatives of solutions to SPDEs  behave badly near the boundary of domains and the $L_p$-norm of the derivatives of solutions cannot be measured without the help of appropriate weights. Interestingly, it turns out that the  weighted spaces $H_{p,\theta}^\gamma(\bR^d_+)$ and $L_p((-\infty,T); H^{\gamma}_{p,\theta}(\bR^d_+))$ are also quite useful in studying deterministic elliptic and parabolic   equations and systems if, for instance, the  free term  $f$ behaves wildly near the boundary as mentioned above,  if systems have lower order derivatives whose coefficients are unbounded near the boundary, or if  systems are defined on non-smooth domains. As an example, if the free term $f$  blows up near the boundary, then  the derivatives of solutions to systems do not belong to  $L_p$-spaces without weights and one needs appropriate weights  to measure the $L_p$-norm of the derivatives of solutions.

We remark that,  if one has a certain unique solvability theory in  weighted Sobolev space 
$L_p((-\infty,T); H^{\gamma}_{p,\theta}(\bR^d_+))$ for systems defined on the half space $\bR^d_+$, then  almost  for free one gets  the corresponding 
 theory in $L_p((-\infty,T); H^{\gamma}_{p,\theta}(\cO))$ for systems defined on $C^1$ domain $\cO\subset \bR^d$ and for any $\gamma\in \bR$.  For  details, we refer to \cite{MR2111792},  where single equations are studied on  $C^1$ domains with the crucial help of the results on a half space. In fact, the result on a half space is not just a starting point, but rather fundamental for the rest of the development of theory.  One noteworthy fact is that, regardless of the regularity parameter $\gamma \in \bR$ especially for large $\gamma$ in dealing high regularity  case, the boundary of the domain is required to be only $C^1$ and no more, which means that the (weighted) regularity of solutions can be improved along the improvement of the forcing term without imposing any further smoothness assumptions on the boundary even if $\gamma$ is large. 

Now, let us place a short description on  related work. The Laplace equation and heat equations in the weighted Sobolev spaces $H_{p,\theta}^\gamma$ setting were first considered in \cite{MR1708104}, where $\theta$ lies in the optimal range $
(d-1,d-1+p)$.
These results were extended to non-divergence type elliptic and parabolic equations with continuous coefficients in \cite{MR2111792}.
Kozlov and Nazarov  in \cite{MR2561181} treated parabolic equations with coefficients depending only on $t$ in mixed space-time norm spaces with the same type of weights.
Recently, in  \cite{MR3318165, MR3147235, MR2990037}  non-divergence and divergence type equations were treated with coefficients having small mean oscillations 
 in both the spatial and time variables. 
In particular, the coefficients in \cite{MR3318165} are further allowed to have no regularity assumptions in the time variable or in one spatial variable. We kindly call the reader's attention to the fact that all the results in \cite{MR3318165, MR3147235, MR2561181, MR2111792, MR2990037, MR1708104} treated only single equations.
Quite recently,  \cite{MR3470413} handled elliptic and parabolic systems in natural modifications of the spaces $H_{p,\theta}^\gamma(\bR^d_+)$ and $L_p((-\infty,T); H^{\gamma}_{p,\theta}(\bR^d_+))$ for matrix-valued $u$ and $f$.

The behind spirit of this paper is the same as the one of  \cite{MR1708104}. This work of \cite{MR1708104} for deterministic equations is the  important preparation for the next step of building a decent theory of stochastic equations on $C^1$-domains. We do this preparation  for stochastic systems in this paper. In fact, there is a preceding work, \cite{MR3470413} with the same purpose. Although it is quite elaborated, we have felt it unsatisfactory.

In this paper we extend the results in \cite{MR3470413}  to a considerably more general setting. 
Compared to the  results in \cite{MR3470413}, the main features of our results can be summarized as follows:
\begin{itemize}
\item Extension on the range of admissible weights:  the condition $\theta\in (d-1,d+1)$ if $p\geq 2$ and $\theta\in (d+1-p,d+p-1)$ if $1<p\leq 2$  in \cite{MR3470413} is extended to the full range $\theta\in (d-1,d-1+p)$.

\item The additional artificial assumption $A^{1j}\equiv 0$ for $j=2,\cdots, d_1$ in  \cite{MR3470413} is dropped in this paper.

\item While $A^{ij}=A^{ij}(t)$s are assumed to depend only on $t$  in \cite{MR3470413}, in this paper $A^{ij}$s depend on $(t,x)$ and they are merely measurable in $t$ and they have small BMO in $x$. 
\end{itemize}
The main reason why, in this paper, we can drop such extra conditions assumed in \cite{MR3470413}  is that we use somewhat  different approaches that we now explain.
The overall procedure, a standard one in $L_p$-theory,  to obtain the main results is  obtaining  a priori estimates and then using the method of continuity.
While in  \cite{MR3470413}  the above extra conditions were needed for attempting the estimation of the sharp functions of the \emph{second derivatives} of solutions, in this article we only estimate the sharp functions of the\emph{ first derivatives}, and then we estimate the weighted $L_p$-norms of solutions and their  second derivatives  from the obtained estimates of the first derivatives and the help of $L_p$-estimates without weights for systems like \eqref{eq_intro_2} through a partition of unity argument.  
Another important remark is that unlike in  \cite{MR3470413} we now use  the Fefferman-Stein (WFS) theorem and the Hardy-Littlewood (WHL) maximal function theorem with appropriate Muckenhoupt weights ($A_p$ weights below) \emph{for the first derivatives} of solutions. Doing so, we can drop those restrictive conditions imposed on the leading coefficient matrices  in \cite{MR3470413} and keep the full range $ (d-1,d-1+p)$ for $\theta$. For the record, in our setting the aforementioned WFS theorem and WHL maximal function theorem with appropriately designed $A_p$ weights, which make us keep the full range of $\theta$, are only  efficient for the first derivatives of solution. The same strategy does not work with  the second derivatives of solutions. In this sense, paying attention to the first derivatives is the optimal strategy for our problem of systems, meaning that we take off most of artificial conditions.

The strategy of estimating the first derivatives is also used in \cite{MR3318165} for single equations. The key step in \cite{MR3318165} and this paper is the estimates of mean oscillations of solutions (see Lemmas \ref{lem0106_1} and \ref{lem0201_1}) in common. While \cite{MR3318165} used a weighted version of mean oscillation estimates, in this paper, thanks to WFS theorem and WHL maximal function theorem for appropriate $A_p$ weights, we only use unweighted mean oscillation estimates of first derivatives, which make us fully use the results for the usual Sobolev spaces without weights.

In addition to the fact that parabolic {\em systems} are considered in weighted Sobolev spaces in this paper, one of main features of our assumptions on the coefficient matrices is that, as we mentioned,  the leading coefficients $A^{ij}(t,x)$ are allowed to have no regularity assumptions other than measurability as functions of $t$ for each fixed $x$ and have small and bounded, not necessarily vanishing on the boundary,  mean oscillations in $x$.
The class of vanishing mean oscillation coefficients ($\text{VMO}_x$ coefficients) was first considered in \cite{MR2304157} for both divergence and non-divergence type equations (not systems) in the usual Sobolev spaces without weights.
Then, there has been considerable progress in efforts to reduce the regularity assumptions on the leading coefficients so that they are allowed to be merely measurable in one spatial direction (in the elliptic case) and in one spatial direction and the time variable (in the parabolic case). A few of good references about this direction of study are \cite{MR2338417, MR2300337, MR2601069, MR2680179, MR2764911}; also, see a survey paper \cite{MR4156495} and references therein.
However, when systems (not equations) in non-divergence form are considered, it is not clear if one can get less regular leading coefficient matrices than those considered in this paper (i.e  coefficient matrices which are measurable in $t$, small $\text{BMO}$ in $x$) even in the usual weighted Sobolev spaces (see \cite{MR2338417}).

Closing the introduction of this paper, we pose a (very) brief history of weighted Sobolev spaces in connection with equations/systems as a landscape in which our paper can be settled. Weighted Sobolev spaces are extensively considered to study elliptic and parabolic equations/systems in a variety of references, where weights are introduced by various reasons.
For example, weights are used to deal with degenerate equations, the singularity of the boundary, and blow-up or very oscillatory coefficients.
See \cite{MR802206}, where motivations of having weights are explained along with various types of weighted Sobolev spaces and their properties, as well as elliptic equations are solved in weighted Sobolev spaces.
In \cite{MR775683} weighted Sobolev spaces are used for domains with wedges.
Parabolic equations are solved in parabolic weighted Sobolev spaces in \cite{MR4072650}.
It is also observed that estimates of solutions in Sobolev spaces with  $A_p$ weights result in the unique solvability of equations in a wider class of function spaces (e.g. estimates in un-mixed normed spaces with weights give the unique solvability in $L_{p,q}$ spaces) due to the extrapolation theorem of Rubio de Francia \cite{MR745140}. 
Also see \cite{MR3812104} for $L_p$-estimates with $A_p$ weights.

Throughout the paper, we impose the Legendre-Hadamard ellipticity condition on the leading coefficients, i.e., there exists a constant $\delta>0$ such that 
\begin{equation}
                                \label{eq_intro_1}
\Re\left(\sum_{i,j=1}^d
\eta^{\text{tr}} \xi_i \xi_j A^{ij}(t,x)\bar\eta\right)
\ge \delta |\xi|^2|\eta|^2
\end{equation}
holds for all $(t,x) \in \bR\times\bR^{d}_+$, $\xi=(\xi_1,\ldots,\xi_d) \in \bR^d$, and 
$\eta=\begin{bmatrix}
\eta^1\\  \vdots \\ \eta^{d_1}\\
\end{bmatrix}$  with $\eta^k\in \bC$, $k=1,\ldots,d_1$, where  $\eta^{\text{tr}}$ denotes the transpose of $\eta$ and
$\Re(f)$ indicates the real part of $f$.  

Also, all through the paper, we assume that  $A^{ij}(t,x)$ are merely measurable in $t$ and have small BMO semi-norm with respect to $x$ (see  Section  \ref{sec_prelim}, Assumption \textbf{A}$(\rho,\varepsilon)$).  We also impose 
the boundedness condition
\begin{eqnarray}
|a^{ij}_{kr}(t,x)| \le \delta^{-1},\quad (t,x)\in \bR\times\bR^d_+\label{eq_intro_3}
\end{eqnarray}
for all $ i,j=1,\ldots,d$, $ k,r=1,\dots,d_1$, where $\delta>0$ is taken  from (\ref{eq_intro_1}).

The paper is organized as follows. In Section 2 we introduce weighted Sobolev spaces and our main result, Theorem \ref{thm_main_1}. In Section 3 we study  systems with coefficients depending only on $t$.  In Section 4 we obtain sharp function estimates (mean oscillation estimates) of solutions.  Finally we prove our main result in Section 5.

We use the following notations.
\begin{itemize}
\item  $D_i=\frac{\partial}{\partial x_i}$, $D_{ij}=\frac{\partial^2}{\partial x_i\partial x_j}$. For a $d_1\times 1$ matrix-valued function $u=[u^1 \cdots u^{d_1}]^{\text{tr}}$  $D_iu(x):=[D_iu^1(x)\cdots D_i u^{d_1}(x)]^{\text{tr}}$. For $u=u(t,x)$ the partial derivative of $u$ with respect to $t$, $u_t$, is understood similarly.  

\item  Throughout the proofs in this paper, the constant $N=N(\cdots)$ depends only on the parameters inside of the parentheses and can be generic along the arguments.

\item We will meet $d_1\times1$ matrix valued, $d_1\times d$ matrix valued, or $d_1\times d\times d$ tensor valued functions $f$ depending on situations.

\item  The norm notation $|A|^2$ of a matrix or tensor denotes the sum of all squares of the components of $A$. For instance, given $u=[u^1\cdots u^{d_1}]^{\text{tr}}$
$$
|u|=\sqrt{\sum_{k=1}^{d_1}|u^k|^2},\quad |Du|=\sqrt{\sum_{k}^{d_1}\sum_{i=1}^d|D_iu^k|^2},\quad |D^2u|=\sqrt{\sum_{k,i,j}|D_{i,j}u^k|^2}.
$$
\end{itemize}

\section{The description of main result}\label{sec_prelim}

In what follows we write the system (\ref{eq_intro_2}) as
\[
-u_t+A^{ij}(t,x)D_{ij}u + B^i(t,x)D_iu+C(t,x)u-\lambda u = f,
\]
assuming  the summations  upon the repeated indices. In Section 3  $A^{ij}$s depend only on  $t$ and we write
\[
-u_t+A^{ij}(t)D_{ij}u+B^iD_iu+C u-\lambda u = f.
\]

Before we state our result Theorem \ref{thm_main_1}, let us first introduce the function spaces that we use in the theorem and we play with through out this paper. 

The basic function spaces are $H_{p,\theta}^\gamma=H_{p,\theta}^\gamma(\bR^d_+)$, where $\gamma\in \bR$, which were introduced  in \cite{MR1708104} for scalar valued functions defined on $\bR^d_{+}$.  The main ingredients of these spaces are the spaces of Bessel potentials defined on $\bR^d$. In this paper we need a $d_1\times 1$ matrix valued function version of this. 

Given $p\in (1,\infty)$ let $L_p=\{f=[f^1 \,\cdots\, f^{d_1}]^{\text {tr}}:\|f\|_{L_p}=\left(\int_{\bR^d}|f|^pdx\right)^{1/p}<\infty\}$. Then for $\gamma\in\bR$ we define the space of Bessel potentials $H^{\gamma}_p$ by
$H^{\gamma}_p=(1-\Delta)^{-\gamma/2}L_p$ as the
set of all matrix valued distributions $u=[u^1\cdots u^{d_1}]^{\text{tr}}$ defined on $\bR^d$ such that $(1-\Delta)^{\gamma/2}u:=[(1-\Delta)^{\gamma/2}u^1\,\cdots \, (1-\Delta)^{\gamma/2}u^{d_1}]^{\text tr}\in
L_p$, i.e.,
$$
\|u\|_{H^{\gamma}_p}=\|(1-\Delta)^{\gamma/2}u\|_{L_p}<\infty,
$$
where $\|(1-\Delta)^{\gamma/2}u\|_{L_p}:=\|\cF^{-1}[\,(1+|\cdot|^2)^{\gamma/2}\cF(u)\,]\,\|_{L_p}$.
Here, the Fourier transform $\cF(u)$ is defined  by
$$
\cF(u)(\xi)=[\cF(u^1)(\xi)\;\cdots\;\cF(u^{d_1})(\xi)]^{\text tr},
$$
where
$$
\cF(u^k)(\xi)=\widetilde{u^k}(\xi)=\frac{1}{(2\pi)^{d/2}}\int_{\bR^d}e^{-i\xi \cdot x}u^k(x) \, dx,\quad k=1,\ldots,d_1.
$$

Now, for $p\in(1,\infty)$ we take and fix a non-negative scalar valued function $\zeta \in C_0^\infty(\bR_+)$ satisfying
\begin{equation}\label{zeta}
\sum_{n = -\infty}^{\infty} \zeta^p\left(e^{s-n}\right) \ge 1
\end{equation}
for all $s \in \bR$. 
Then for  $\gamma, \theta \in \bR$
we define our basic spaces $H_{p,\theta}^\gamma=H_{p,\theta}^\gamma(\bR^d_+)$  as the set of all matrix valued functions (or distributions) $u=[u^1\;\cdots\;u^{d_1}]^{\text{tr}}$ on $\bR^d_+$ satisfying
$$
\| u \|_{H_{p,\theta}^\gamma}^p
:= \sum_{n = -\infty}^{\infty} e^{n \theta} \|\zeta(\pi(\cdot))u(e^n \cdot) \|_{H^{\gamma}_p}^p < \infty,
$$
where $\pi(x)=\pi(x_1,x')=x_1$; we view $ \zeta(x_1)u(e^n x)$ as a matrix valued function  defined on the whole space $\bR^d$ thank to the fact that $\zeta$ has compact support in $\bR_+$.

 If $\gamma$ is a non-negative integer, due to the choice of $\zeta$ satisfying \eqref{zeta}
the following characterization is available (\cite{MR1708104}):
$$
H_{p, \theta}^{\gamma}=H_{p, \theta}^{\gamma}(\bR^d_+)
= \{ u : x_1^{|\alpha|} D^{\alpha} u \in L_{p, \theta}\  \
\forall \alpha : 0 \le |\alpha| \le \gamma \},
$$
where $L_{p,\theta}=L_{p, \theta}(\bR^{d}_{+})$ is
the weighted $L_p$-space of matrix valued  functions $f=[f^1\;\cdots\;f^{d_1}]^{\text{tr}}$ on $\bR^d_+$ satisfying 
 $ \|f\|_{L_{p,\theta}}:=\left(\int_{\bR^d_+}|f(x)|^px_1^{\theta-d}dx\right)^{1/p}<\infty$. We will denote  $M^{k} f \in L_{p,\theta}$, $k\in\bZ$ if $x_1^k f \in L_{p,\theta}$.
We record that the operators $MD_i$s  and $D_iM$s, $i=1,\ldots,d$, are bounded operators from $H_{p,\theta}^\gamma$ to $H_{p,\theta}^{\gamma-1}$ after using the corresponding lemma in \cite{MR1708104}, which deals the spaces of scalar valued functions.

For the forcing term, the solution, and their derivatives in our parabolic system, we first define the function spaces
$$
\bL_{p,\theta}((S,T)\times \bR^d_+) = L_p \big( (S,T)\times \bR^d_+;x_1^{\theta-d}dx\,dt \big)
$$
for $-\infty \le S < T \le \infty$. The functions $f=[f^1\;\cdots\;f^{d_1}]^{\text tr}$ in this space satisfy
$$
 \|f\|_{\bL_{p,\theta}((S,T)\times \bR^d_+)}=\left(\int^T_S\int_{\bR^d_+}|f(x)|^px_1^{\theta-d}dx\,dt\right)^{1/p}<\infty.
$$

In particular, if $\theta=d$,  the weight disappears and  $\bL_{p,\theta}((S,T)\times \bR^d_+) = L_p((S,T)\times \bR^d_+)$. Similarly as before, we denote  $M^{k} f \in \bL_{p,\theta}((S,T)\times \bR^d_+)$, $k\in\bZ$  if $x_1^{k} f \in \bL_{p,\theta}((S,T)\times \bR^d_+)$.

Then we design our solution space as follows. We write $u \in \fH_{p,\theta}^2((S,T)\times\bR^d_+)$ if
$$
M^{-1} u, \,\,
D_iu,  \,\,
M D_{ij} u, \,\, M u_t \in \bL_{p,\theta}((S,T)\times\bR^d_+), \quad i,j=1,\ldots,d
$$
and  define the norm of $u$ by
\begin{equation}\label{eq20210619_4}
\|u\|_{\fH_{p,\theta}^2((S,T)\times\bR^d_+)} = \|M^{-1} u \|_{p,\theta} + \|Du \|_{p,\theta} + \|M D^2 u \|_{p,\theta} + \|Mu_t\|_{p,\theta},
\end{equation}
where $\|\cdot\|_{p,\theta}$ abbreviates $\|\cdot\|_{\bL_{p,\theta}((S,T)\times\bR^d_+)}$; recall our notations $|Du|$ and $|D^2u|$.


Next, for the convenience of our arguments in this paper, given any  time-space domain $\cD\subset \bR\times \bR^d_+$ we define  $\bW^{1,2}_p(\cD)$ as  the space of matrix valued functions  $u = [u^1 \cdots u^{d_1}]^{\text{tr}}$ defined on $\cD$ satisfying
$$
u, \,\, D_iu, \,\, D_{ij}u, \,\, u_t\in L_p(\cD),  \quad i,j=1,\ldots,d
$$
and $C^{\infty}_0(\cD)$  as the space of infinitely differentiable $d_1\times 1$ matrix valued functions with compact support in $\cD$;  $\cD$ is not necessarily open.

We also define the parabolic H\"{o}lder spaces $C^{\alpha/2,\alpha}(\cD)$, $\alpha\in(0,1)$, as the set of matrix valued functions $f$ defined on $\cD$ satisfying
$$
\|f\|_{C^{\alpha/2,\alpha}(\cD)}:=\sup_{(t,x)\in\cD}|f(t,x)|+\sup_{(t,x)\ne (s,y)\in\cD}\frac{|f(t,x)-f(s,y)|}{|t-s|^{\alpha/2}+|x-y|^{\alpha}}<\infty,
$$
where each magnitude $|\cdot|$ is well understood.

Now, let us explain our  condition for the leading coefficients $A^{ij}$s.  We  frequently use the following balls, cylinders, and parabolic cylinders.  Recall the notation $x=(x_1,x')\in\bR^d_+=\bR_+\times\bR^{d-1}$, where $x'=(x_2,\ldots,x_d)$.  We define
$$
B_r'(x')=\{ y'\in \bR^{d-1}\,|\,|y'-x'|<r\},\quad Q_r'(t,x')=(t-r^2,t)\times B_r'(x'),
$$
$$
B_r(x)=(x_1-r, x_1+r) \times B'_r(x'),
\quad Q_r(t,x)=(t-r^2,t)\times B_r(x),
$$
$$
B_r^+(x)=B_r(x)\cap \bR^d_+,\quad Q_r^+(t,x)=(t-r^2,t)\times B_r^+(x).
$$
We note that the volume of the cylinder $B_r(x)$ in $\bR^d$  and the volume of the \emph{ball} $\{x\in\bR^d:|x|<r\}$ are comparable and we use $B_r(x)$ for the convenience of working with $x_1$ coordinate.
%

For  matrix valued  functions $g$ defined on $\bR\times\bR^{d}_+$ and any fixed $t\in\bR$  we define the average of $g(t,\cdot)$ over $B_r(x)$ by
$$
\dashint_{B_r(x)} g(t,z) \, dz:=\frac{1}{|B_r(x)|}\int_{B_r(x)} g(t,z) \, dz,
$$
where $|B_r(x)|$ is the Lebesgue measure of the cylinder $B_r(x)$.  Then we define
$$
[g(t,\cdot)]_{B_r(x)} = \dashint_{B_r(x)} \bigg| g(t,y) -\dashint_{B_r(x)} g(t,z) \, dz \bigg| \, dy,
$$
which measures the deviation of $g(t,\cdot)$ from the average on $B_r(x)$.  These averaging jobs are considered for each fixed $t$.

 Using them, for any $(s,y)\in\bR\times\bR^d_+$  and $r<y_1$,  we define {\bf{the mean oscillation}} of $g$ over $Q_r(s,y)=(s-r^2,s)\times(y_1-r, y_1+r) \times B'_r(y')= Q^+_r(s,y)$ with respect to the spatial variables as
$$
\text{osc}_{\sfx} \left(g,Q_r(s,y)\right)
:= \frac{1}{r^2}\int_{s-r^2}^{\,\,s}
\left[ g(\tau, \cdot)\right]_{B_r(y)} \, d\tau.
$$
Finally, for $\rho\in (1/2,1)$, we denote
$$
g^{\sfx,\#}_\rho :=\sup_{(s,y) \in \bR\times\bR^{d}_+}\;\;\sup_{r\in (0, \rho y_1]}\text{osc}_\sfx\left(g,Q_r(s,y)\right).
$$
Applying these notations to the diffusion coefficient matrices $A^{ij}$,  $i,j=1,\ldots,d$, in place of $g$, let us state the following assumption for coefficient matrices $A^{ij}$s, $B^i$s, and $C$. 


\begin{AssumpA}
							\label{assum0309_1}
For $\rho\in (1/2, 1)$ and $\varepsilon\in (0,1)$, we have the following bounded mean oscillation (BMO) condition for $A^{ij}$s and the boundedness conditions for $B^i$s, $C$:
$$
\sum_{i,j=1}^d(A^{ij})_{\rho}^{\sfx,\#} +\sup_{t,x}(|MB^i|+|M^2C|)\le \varepsilon.
$$
\end{AssumpA}
Obviously Assumption \textbf{A}$(\rho,\varepsilon)$  holds for any $\rho, \varepsilon>0$ if $A^{ij}$ depend only on $t$ and $B^i=C=0$.

Now, we are ready to state the main theorem of the paper.

\begin{theorem}[Weighted $L_p$-theory on a half space]\label{thm_main_1}
Let $T\in(-\infty,\infty]$, $\lambda\ge 0$, $p\in (1,\infty)$ and $\theta \in (d-1,d-1+p)$.  Then there exist positive constants $\rho\in (1/2,1)$ and $\varepsilon$, depending only on $d$, $d_1$, $\delta$, $p$, and $\theta$,  such that, under Assumption \textbf{A}$(\rho,\varepsilon)$, for any $u \in \fH_{p,\theta}^2((-\infty,T)\times\bR^d_+)$ satisfying the system
\begin{equation}\label{eq_main_2}							
- u_t + A^{ij}(t,x) D_{ij} u +B^i(t,x)D_i u+C(t,x)u- \lambda u = f
\end{equation}
in $(-\infty,T) \times \bR^d_+$ with $Mf \in \bL_{p,\theta}((-\infty,T)\times \bR^d_+)$, we have the estimate
\begin{eqnarray}
\lambda \|Mu\|_{p,\theta}+\sqrt{\lambda} \|MDu\|_{p,\theta}+\|u\|_{\fH^2_{p,\theta}}\leq N\|Mf\|_{p,\theta}\label{eq_main_1}
\end{eqnarray}
where $\|\cdot\|_{p,\theta}=\|\cdot\|_{\bL_{p,\theta}((-\infty,T)\times \bR^d)}, \|\cdot\|_{\fH^2_{p,\theta}}= \|\cdot\|_{\fH^2_{p,\theta}((-\infty,T)\times \bR^d)}$, and $N$ depends only on $d,d_1,\delta,p,\theta$, and $\rho$. Moreover, for any $f$ satisfying $M f \in \bL_{p,\theta}((-\infty,T) \times \bR^d_+)$, there exists a unique solution $u \in \fH_{p,\theta}^2((-\infty,T) \times \bR^d_+)$ to the system  (\ref{eq_main_2}).
\end{theorem}

\begin{remark}
The range of $\theta$ in Theorem  \ref{thm_main_1} is sharp. If $\theta\not \in (d-1,d-1+p)$, then the theorem does not hold even for the heat equation. See \cite{MR1708104} for an explanation. We point out that our result for systems, Theorem \ref{thm_main_1}, preserves this sharp range of $\theta$. 
\end{remark}

\section{Systems with coefficients depending only on $t$}
In this section  all $A^{ij}$s  depend only on $t$ and  are merely measurable. We consider the system
$$
-u_t+A^{ij}(t)D_{ij}u-\lambda u =f.
$$

 Let us recall 
$\mathbb{W}_p^{1,2}\left((-\infty,T) \times \Omega\right)$, the Sobolev space  of the $d_1\times 1$ matrix valued functions  $u$ on $(-\infty,T) \times \Omega$ satisfying 
$$
\|u\|_{\mathbb{W}_p^{1,2}\left((-\infty,T) \times \Omega\right)}\\
:=\|D^2u\|_p+\|Du\|_p+\|u\|_p+\|u_t\|_p <\infty,
$$
where $\|\cdot\|_p=\|\cdot\|_{\bL_p((-\infty,T)\times \Omega)}$.

\begin{proposition}[Unweighted $L_p$-theory on the whole space or a half space]
                                                         \label{pro entire}
Let  $T\in (-\infty, \infty]$, $\lambda \geq0$,  $p\in(1,\infty)$, and $\Omega = \bR^d$ or $\Omega = \bR^d_+$.  Then for any  $u \in \mathbb{W}_p^{1,2}\left((-\infty,T) \times \Omega\right)$ satisfying   the system
\begin{equation}\label{eq_time_4}
-u_t+A^{ij}(t)D_{ij}u-\lambda u =f, \quad (t,x)\in (-\infty, T)\times \Omega
\end{equation}
and  $u(t,0,x') = 0$ in the case of $\Omega = \bR^d_+$,  where $f\in \bL_p((-\infty,T)\times \Omega)$,  we have the estimate
\begin{equation}  \label{eq0119_1}
\lambda \|u\|_{p}+\sqrt{\lambda}\|Du\|_{p}
+\|D^2u\|_{p}+\|u_t\|_{p}\leq N\|f\|_{p},
\end{equation}
where  $N$ depends only on $d,d_1, \delta, p$.
Moreover, for any $\lambda > 0$ and $f \in \bL_p\left((-\infty,T) \times \Omega\right)$, 
there exists a unique $u \in \bW_p^{1,2}\left((-\infty,T) \times \Omega\right)$ satisfying the system \eqref{eq_time_4}  and the Dirichlet condition $u(t,0,x') = 0$ in the case of $\Omega = \bR^d_+$.
\end{proposition}

\begin{proof}
This proposition can be derived from \cite[Theorem 2, Theorem 4]{MR2771670}, where the results are proved for higher order systems (including second order systems) with $\lambda \geq \lambda_0$ for some $ \lambda_0\geq 0$  under the conditions that $A^{ij}$ are measurable in $t$ and have small mean oscillations in $x$.

If $A^{ij}$ depend only on $t$, the mean oscillations in $x$ vanish and 
one can be free from the restriction $\lambda\ge\lambda_0$ relying on  the usual scaling argument.
Indeed, if $\lambda_0>0$ and $\lambda \in (0,\lambda_0)$, then we set $R = \lambda_0/\lambda$ and consider the vector-valued function 
$$
\tilde{u}(t,x) = R^{-1} u(Rt, \sqrt{R}x).
$$
This function satisfies the system
$$
-\tilde{u}_t + A^{ij}(Rt) D_{ij} \tilde{u} - \lambda_0 \tilde{u} = \tilde{f}
$$
on $(-\infty, T/R) \times \Omega$, where $\tilde{f}(t,x) = f(Rt, \sqrt{R}x)$.
Then, since the coefficients $A^{ij}(Rt)$ satisfy the same conditions as $A^{ij}(t)$,  \cite[Theorem 2, Theorem 4]{MR2771670} can be applied and  we obtain the estimate
$$
\lambda_0 \|\tilde{u}\|_{p}+\sqrt{\lambda_0}\|D\tilde{u}\|_{p}
+\|D^2\tilde{u}\|_{p}+\|\tilde{u}_t\|_{p}\leq N\|\tilde{f}\|_{p},
$$
where $\|\cdot\|_p = \|\cdot\|_{L_p\left((0,R^{-1}T)\times\Omega\right)}$. 
We note that the constant $N$ is independent of the upper limit of the time interval.
Then we scale back to $u$ and have \eqref{eq0119_1}.
Now, for small constant $\varepsilon > 0$ we write 
$$
-u_t + A^{ij}(t) D_{ij}u - \varepsilon u = f - \varepsilon u
$$
on $(-\infty,T) \times \Omega$. Applying the estimate just proved for small $\lambda > 0$, we have
$$
\varepsilon \|u\|_{p}+\sqrt{\varepsilon}\|Du\|_{p}
+\|D^2u\|_{p}+\|u_t\|_{p}\leq N\|f\|_{p} + N \varepsilon\|u\|_p.
$$
Hence,  letting $\varepsilon \searrow 0$,  we obtain \eqref{eq0119_1} for the case $\lambda=0$.

Then for the second statement of our proposition we follow the routine based on the a priori estimate \eqref{eq0119_1}, the method of continuity, and the unique solvability of the system consisting of $d_1$ independent (not mixed) heat equations.
\end{proof}

 Proposition \ref{pro entire}  leads us to the following lemma that involves weights near the boundary of the half space. We recall the definition of the space $\fH^2_{p,\theta}((-\infty,T)\times \mathbb{R}^d_+)$  from Section  \ref{sec_prelim}. 
\begin{lemma}
   \label{lem half}
Let $T\in (-\infty, \infty]$, $\lambda \geq 0$,  $p>1$, and $\theta\in (d-p,\infty)$. Then for any   $u\in \fH^2_{p,\theta}((-\infty,T)\times \bR^d_+)$  satisfying the system
$$
-u_t+A^{ij}(t)D_{ij}u-\lambda u =f
$$
on $ (-\infty, T)\times \bR^d_+$,  where $Mf \in \mathbb{L}_{p,\theta}((-\infty,T)\times \mathbb{R}_+^d)$, we have the estimate
\begin{eqnarray}
 \lambda \|Mu\|_{p,\theta}+ \sqrt{\lambda} \|MDu\|_{p,\theta} + \|u\|_{\fH^2_{p,\theta}}
\leq
N(\|M^{-1}u\|_{p,\theta}+\|Mf\|_{p,\theta}) \label{eq0129_6},
\end{eqnarray}
 where   $\|\cdot\|_{p,\theta}=\|\cdot\|_{\bL_{p,\theta}((-\infty,T)\times \bR^d_+)}$ and $N$ depends only on  $d,d_1,\delta, p$ and $\theta$.
The same conclusion holds if $\theta \in (d-1, d-1+p)$,  $u \in \bW_p^{1,2}\left((-\infty,T) \times \bR^d_+\right)$,  $u(t,0,x') = 0$, and $Mf \in \mathbb{L}_{p,\theta}((-\infty,T)\times \mathbb{R}_+^d)$.
In this case, $u \in \fH_{p,\theta}^2\left((-\infty,T) \times \bR^d_+\right)$.
\end{lemma}
\begin{proof}
We can prove this lemma by following the proof of  Lemma 2.2 in \cite{MR1690093}.  Also see for instance Theorem 3.5 in \cite{MR3318165}.  These deal with single equations of course and  we need a version for  systems. The idea of the proof is delightful and for the reader's convenience we provide a proof  below.

1. We intend to use Proposition \ref{pro entire} to pursuit the weighted norms. For this we start with the following.
Take and fix a function $\zeta=\zeta(s)\in C^{\infty}_0(\bR_+)$ satisfying
$$
\int^{\infty}_0 |\zeta(s)|^ps^{-p-\theta+d-1} ds =1,
$$
and for any index $r>0$ we define the function $\zeta_r(x_1):=\zeta(rx_1)$ for $x_1>0$. Then for any matrix-valued  functions $g$ defined on $\mathbb{R}^d_+$,  by Fubini's theorem and change of variables, the following three hold: 
$$
\int^{\infty}_0\int_{\bR^d_+}|\zeta_r(x_1)g(x)|^pdx\; r^{-p-\theta+d-1}dr=\int_{\bR^d_+}|x_1g(x)|^px_1^{\theta-d}dx,
$$
\begin{eqnarray*}
&&\int_0^\infty \int_{\bR^d_+} \left| \zeta'_r(x_1)  g(x)\right|^p \, dx \, r^{-p - \theta + d - 1} \, dr
\\&
=& \int_0^\infty\int_{\bR^d_+} |r \zeta'(r x_1) g(x)|^p \, dx \, r^{-p-\theta+d-1} \, dr = N\int_{\bR^d_+}|g(x)|^px_1^{\theta-d}dx,
\end{eqnarray*}
where
$$
N = N(d,p,\theta) = \int_0^\infty |\zeta'(s)|^p s^{-\theta+d-1} \, ds,
$$
and
$$
\int^{\infty}_0\int_{\bR^d_+}|\zeta''_r(x_1)  g(x)|^pdx \;r^{-p-\theta+d-1}dr= N\int_{\bR^d_+}|x^{-1}_1g(x)|^px_1^{\theta-d}dx,
$$
where in this case
$$
N = N(d,p,\theta) = \int_0^\infty |\zeta''(s)|^p s^{-p-\theta+d-1} \, ds.
$$
Utilizing the idea of the computation, we also have
$$
\int^{\infty}_0\int_{\bR^d_+}|r\zeta^{(n)}_r(x_1)  g(x)|^pdx \;r^{-p-\theta+d-1}dr= N\int_{\bR^d_+}|x^{-n}_1g(x)|^px_1^{\theta-d}dx,\; n=0,1
$$
and
$$
\int^{\infty}_0\int_{\bR^d_+}|r^{-1}\zeta^{(n)}_r(x_1)  g(x)|^pdx \;r^{-p-\theta+d-1}dr= N\int_{\bR^d_+}|x^{2-n}_1g(x)|^px_1^{\theta-d}dx,\;
n=1,2,3.
$$

We note that the integrals on $\bR^d_+$ are the same as the integrals on $\bR^d$ once we extend the matrix-valued functions inside of $|\cdot|^p$s to the whole space by zero matrices. 

2.  As we have just mentioned, using $\zeta_r$ defined in step 1, we regard  $\zeta_r(x_1)u(t,x)$ as a matrix valued function  defined on $(-\infty,T)\times \bR^d$ by extending  $\zeta_ru$   to be $d_1\times 1$ zero matrix  on $(-\infty,T)\times \{x=(x_1,x')\in\bR^d:x_1\le 0\}$.
Recalling the summation rule upon the repeated indices, we observe that the matrix-valued function $\zeta_r u$ satisfies
\begin{eqnarray}
							\label{eq_time_1}
&&-(\zeta_r u)_t+A^{ij}(t)D_{ij}(\zeta_ru)-\lambda \zeta_ru
\nonumber\\
&=&\zeta_rf+A^{i1}(t)\zeta'_rD_iu+ A^{1j}(t)\zeta'_rD_ju+A^{11}(t)\zeta''_r u
\end{eqnarray}
as
\begin{eqnarray}
&&D_j(\zeta_r u)=\zeta_r D_ju+\textbf{1}_{j=1}\zeta'_ru,\nonumber\\
 &&D_{ij}(\zeta_r u)=\zeta_r D_{ij}u+\textbf{1}_{i=1}\zeta'_r D_j u+\textbf{1}_{j=1}\zeta'_r D_i u+ \textbf{1}_{i=1}\textbf{1}_{j=1}\zeta''_ru \label{eq_time_2}
\end{eqnarray}
for each $i,j=1,\ldots,d$, and $\zeta_r$ is a function of $x_1$, where $\textbf{1}_{j=1}=1$ if $j=1$ and $0$ otherwise.  Since  the compact support of $\zeta_r$ is located away from $x_1=0$,  $u\in \fH^2_{p,\theta}((-\infty,T)\times \mathbb{R}^d_+)$ implies  $\zeta_r u\in \bW^{1,2}_p((-\infty,T)\times \bR^d)$. Then  (\ref{eq_time_1}) with the observation that the right hand side of (\ref{eq_time_1}) is in $\bL_p((-\infty,T)\times \mathbb{R}^d)$,  the condition \eqref{eq_intro_3} and Proposition \ref{pro entire} lead us to
\begin{multline}
							\label{eq0129_3}
\lambda^p \|\zeta_r u\|^p_p+\lambda^{p/2}\|D(\zeta_r u)\|^p_p+\|D^2(\zeta_r u) \|^p_p+\| (\zeta_r u)_t\|^p_p
\\
\leq  N \left(\|\zeta f\|^p_p+\|\zeta'_rDu\|^p_p+\|\zeta''_ru\|^p_p\right),\end{multline}
where $\|\cdot\|_p=\|\cdot\|_{L_p((-\infty,T)\times \bR^d)}$ and $N=N(d,d_1,\delta,p)$.
From \eqref{eq0129_3} and the relation \eqref{eq_time_2}, we obtain that
\begin{equation*}
\lambda^p \|\zeta_r u\|_p^p + \|\zeta_r D^2u\|_p^p + \|\zeta_r u_t\|_p^p \leq N \left(\|\zeta f\|^p_p+\|\zeta'_rDu\|^p_p+\|\zeta''_ru\|^p_p\right).
\end{equation*}
Then using this estimate along with \eqref{eq0129_3} and the relations derived from \eqref{eq_time_2}
$$
\zeta_r' D_1 u = \frac{1}{2}\left( D_{11}(\zeta_r u) - \zeta_r D_{11} u - \zeta_r'' u \right),
$$
$$
\zeta_r' D_ju = D_{1j}(\zeta_r u) - \zeta_r D_{1j}u, \quad j \neq 1,
$$
we have
$$
 \lambda^p \|\zeta_r u\|_p^p+\|\zeta_r' Du\|_p^p+ \|\zeta_r D^2u\|_p^p + \|\zeta_r u_t\|_p^p \leq N \left(\|\zeta_r f\|^p_p+\|\zeta'_rDu\|^p_p+\|\zeta''_ru\|^p_p\right).
$$
Now, multiplying both sides  of this inequality by $r^{-p-\theta+d-1}$,  integrating with respect to $r$ over $(0,\infty)$, and using step 1, we get
\begin{multline}
\label{eq_almost}
\lambda^p \|Mu\|^p_{p,\theta}+\|Du\|_{p,\theta}^p +\|MD^2u\|^p_{p,\theta}+\|Mu_t\|^p_{p,\theta}
\\
\leq N\left(\|Mf\|^p_{p,\theta}+\|Du\|^p_{p,\theta}+|M^{-1}u\|^p_{p,\theta}\right),
\end{multline}
where  $N=N(d,d_1,\delta,p,\theta)$. To arrive at  \eqref{eq0129_6} from here, on the one hand we bring in the interpolation inequality (see \cite[Lemma 3.3]{MR3318165}),
\begin{equation}\label{eq20210624_1}
\sqrt{\lambda}\|M Du\|_{p,\theta} \leq N \lambda \|M u\|_{p,\theta} + N \|M D^2u\|_{p,\theta}
\end{equation}
which holds for $\theta -d+p> 0$ i.e. $\theta>d-p$,
where $N$ is a universal constant (independent of $d$, $u$, $p$, and $\theta$).
On the other hand we dominate the term $\|Du\|^p_{p,\theta}$ on the right hand side of \eqref{eq_almost} by another interpolation
\begin{equation}\label{eq210608_1}
\|Du\|^p_{p,\theta}\le \varepsilon (N\|MD^2u\|^p_{p,\theta}+ N \|Du\|^p_{p,\theta})+N\|M^{-1} u\|^p_{p,\theta}
\end{equation}
for any $\varepsilon>0$, where the first two $N$s do not depend on $\varepsilon$.
Indeed, by identity  $\zeta'_r D_ju=D_j(\zeta'_r u)-\mathbf{1}_{j=1}\zeta''_r u$, the usual interpolation of Sobolev norms, and step 1, we have
\begin{eqnarray*}
\|Du\|^p_{p,\theta}&=&\int^{\infty}_0 \|\zeta'_r Du\|^p_p\; r^{-p-\theta+d-1}dr\\
&\le &\int^{\infty}_0\left(\varepsilon r^{-p}\|D^2(\zeta'_r u)\|^p_p+N\varepsilon^{-1}r^p\|\zeta'_r u\|^p_p\right) \; r^{-p-\theta+d-1}dr\\
&& +N \|M^{-1}u\|^p_{p,\theta}
\end{eqnarray*}
and, expressing $D^2(\zeta'_r u)$ in terms of $\zeta'''_r u$, $\zeta''_rD_ju$, $\zeta'_rD_{ij}u$ and using all the characterizations prepared in step 1, we obtain  \eqref{eq210608_1}. In this inequality all the terms are finite and we can move $\varepsilon\|Du\|^p_{p,\theta}$ to the left side of \eqref{eq_almost} with sufficiently small $\varepsilon$. Doing so, we have \eqref{eq0129_6}.

The assertions when $u \in \bW_p^{1,2}\left((-\infty,T) \times \bR^d_+\right)$  and  $u(t,0,x') = 0$ follow from the same lines of the proof  once we confirm $\|M^{-1}u\|_{p,\theta} < \infty$.
To check this, we first note that
$$
\|M^{-1}u\|_{p, \theta} \leq \|M^{-1}u I_{x_1 \in (0,1)}\|_{p,\theta} + \|u I_{x_1 \geq 1}\|_p.
$$
Since we have zero Dirichlet condition and the condition  $d-1<\theta < d -1 + p$,  Hardy's inequality can be applied twice with the two sides of the restriction of $\theta$ and we have
$$
\|M^{-1}u I_{x_1 \in (0,1)}\|_{p,\theta}
\leq N\|Du I_{x_1\in(0,1)}\|_{p,\theta}\le N\|MD^2u I_{x_1\in(0,1)}\|_{p,\theta},\; N=N(d,d_1,p,\theta)
$$
if $\|MD^2u I_{x_1\in(0,1)}\|_{p,\theta}<\infty$ and this gives $\|M^{-1}u\|_{p,\theta} < \infty$. The finiteness of $\|MD^2u I_{x_1\in(0,1)}\|_{p,\theta}$ follows from the inequality
$$
\|MD^2u I_{x_1\in(0,1)}\|_{p,\theta}\leq N\|D^2u\|_{p,\theta}
$$
and the condition  $u \in \bW_p^{1,2}\left((-\infty,T) \times \bR^d_+\right)$.
The lemma is proved.
\end{proof}

Comparing the a priori estimate \eqref{eq0129_6} with our aim, Theorem \ref{thm_main_1},  we are about to remove the term $\|M^{-1}u\|_{p,\theta}$  from the estimate \eqref{eq0129_6}.  This job is quite involved and in fact the rest of our paper works on this.  Recall that in this section we assume that the matrices $A^{ij}$s depend only on $t$ and have zero oscillation  with respect to the space variables. Under such conditions,  the job can be done relatively easily when $p=2$ (makes everything beautiful), $\theta=d$ (gives $x_1^{\theta-d}=1 $, no weight). The following proposition and the theorem work on this. 
\begin{proposition}
							\label{prop0925_1}
Let $T \in (-\infty, \infty]$, $\lambda \ge 0$. Assume that  $u \in C_0^{\infty}((-\infty,T] \times \bR^d_+)$ if $T<\infty$ and $u \in C_0^{\infty}((-\infty,\infty) \times \bR^d_+)$ if $T=\infty$. Denote  
\begin{equation}
							\label{eq0919_1}
f:=- u_t + A^{ij}(t) D_{ij} u - \lambda u.
\end{equation}
Then  $Mf$ belongs to $L_2((-\infty,T)\times \bR^d_+)=\bL_{2,d}((-\infty,T)\times \bR^d_+)$ and we have the estimate
\begin{equation}
\|M^{-1}u\|_{2}\le N \|Mf\|_{2}\label{eq0925_1}
\end{equation}
where $N = N(\delta)$ and  $\|\cdot\|_{2}=\|\cdot\|_{L_{2}((-\infty,T)\times \bR^d_+)}$.
\end{proposition}

\begin{proof}
1. Let $T<\infty$. Denote $\bar{u}=[\bar{u}^1\cdots \bar{u}^{d_1}]^{\text{tr}}$, where $\bar{u}^k$ is the complex conjugate function of $u^k$.
Performing a left multiplication of $1\times d_1$ matrix-valued function  $- \bar{u}^{\text{tr}}$ on both sides of  \eqref{eq0919_1} and integrating them over $(-\infty, T)\times \bR^d_+$, we have
\begin{equation}
							\label{eq0105_01}
\int^T_{-\infty}\int_{ \bR^d_+} \left(\bar{u}^{\text{tr}} u_t -  \bar{u}^{\text{tr}} A^{ij}(t) D_{ij} u + \lambda \bar{u}^{\text{tr}} u \right) \, dx \, dt = -\int^T_{-\infty}\int_{ \bR^d_+} \bar{u}^{\text{tr}} f \, dx \, dt.
\end{equation}
We will take the real part of \eqref{eq0105_01} and use  Legendre-Hadamard ellipticity condition \eqref{eq_intro_1} that we imposed.  For this we first note that
$$
\int^T_{-\infty}\int_{ \bR^d_+} \left( \overline{u_t}^{\text{tr}} u + \bar{u}^{\text{tr}} u_t \right) \, dx \, dt =\int^T_{-\infty}\int_{ \bR^d_+} \left( \bar{u}^{\text{tr}} u \right)_t \, dx \, dt
=\int_{ \bR^d_+} |u|^2 (T,x) \, dx
$$
by the fundamental theorem of calculus and hence we have
$$
\Re \left(\int^T_{-\infty}\int_{ \bR^d_+} \bar{u}^{\text{tr}} u_t \, dx \,dt\right) = \frac{1}{2} \int_{\bR^d_+} |u|^2 (T,x) \, dx.
$$
Next, as $u \in C_0^{\infty}((-\infty,T] \times \bR^d_+)$, by integration by parts and the fact that $a^{ij}_{kr}$s depend only on $t$ the integral
\begin{eqnarray}
-\int^T_{-\infty}\int_{ \bR^d_+} \bar{u}^{\text{tr}} A^{ij}(t) D_{ij} u \, dx \, dt
= - \sum_{i,j=1}^d \sum_{k,r=1}^{d_1} \int^T_{-\infty}\int_{\bR^d_+} \overline{u^k} a^{ij}_{kr}(t) D_{ij} u^r \, dx \, dt \nonumber
\end{eqnarray}
becomes
\begin{eqnarray}
\sum_{i,j=1}^d\sum_{k,r=1}^{d_1} \int^T_{-\infty}\int_{\bR^d_+} \overline{D_i u^k}  a_{kr}^{ij}(t) D_j u^r \, dx \, dt= \int^T_{-\infty}\int_{ \bR^d_+} \overline{D_i u}^{\text{tr}} A^{ij}(t) D_j u \, dx \, dt.\nonumber
\end{eqnarray}
Again, since $u \in C_0^\infty\left( (-\infty,T] \times \bR^d_+\right)$,  if we extend $u$ to be zero in the domain $(-\infty,T) \times\{x = (x_1,x') \in \bR^d: x_1 \le  0\}$, then  the extension of $u$, still denoted by $u$, belongs to $C_0^\infty\left( (-\infty,T] \times \bR^d\right)$.  Now, Plancherel's formula,  the condition \eqref{eq_intro_1},  and  Parseval's identity give
 \begin{eqnarray}
\Re\left(\int^T_{-\infty}\int_{\bR^d} \overline{D_i u}^{\text{tr}} A^{ij}(t) D_j u \, dx \, dt\right)
&=&\Re\left(\sum^d_{i,j=1}\int^T_{-\infty}\int_{\bR^d}\overline{\widetilde{D_iu}}^{\text{tr}}    A^{ij}(t) \widetilde{D_ju} \, d\xi \, dt\right)\nonumber\\
&=&\Re\left(\sum^d_{i,j=1}\int^T_{-\infty}\int_{\bR^d}\bar{\tilde{u}}^{\text{tr}}   \xi_i \xi_jA^{ij}(t) \tilde{u} \, d\xi \, dt\right)\nonumber
\\
&\ge&  \delta \int^T_{-\infty}\int_{\bR^d} |\xi|^2 |\tilde{u}|^2 \, d\xi \, dt \nonumber\\
&=& \delta  \int^T_{-\infty}\int_{\bR^d} |D u|^2 \, dx \, dt
= \delta  \int^T_{-\infty}\int_{\bR^d_+} |D u|^2 \, dx \, dt.\nonumber
\end{eqnarray}

Considering the real parts of \eqref{eq0105_01},  we get
\begin{eqnarray}
&&\delta\int^T_{-\infty}\int_{\bR^d_+} |Du|^2 \, dx \, dt\nonumber\\
&\le& \frac{1}{2}\int_{\bR_+^d} |u|^2(T,x) \, dx + \delta\int^T_{-\infty}\int_{\bR^d_+} |Du|^2 \, dx \, dt + \lambda \int^T_{-\infty}\int_{\bR^d_+} |u|^2 \, dx \, dt \nonumber\\
&\le& \Re \left( -\int^T_{-\infty}\int_{\bR^d_+} \bar{u}^{\text{tr}} f \, dx \, dt \right)\nonumber\\
&\le& \frac{\varepsilon}{2}\int^T_{-\infty}\int_{\bR^d_+} | x_1^{-1}u|^2 \, dx \, dt + \frac{1}{2\varepsilon}\int^T_{-\infty}\int_{\bR^d_+} | x_1 f|^2 \, \, dx \, dt \label{eq_time_3}
\end{eqnarray}
for any $\varepsilon > 0$, where  the last inequality follows from
$$
2\left| \bar{u}^{\text{tr}}(t,x) f(t,x) \right| \le \varepsilon |x_1^{-1}u(t,x)|^2 \,  + \frac{1}{\varepsilon} | x_1f(t,x)|^2 \,.
$$
Furthermore,  we note that Hardy's inequality tells
$$
\int^T_{-\infty}\int_{\bR^d_+} |x_1^{-1}u|^2  \, dx \, dt
\le 2^2\int^T_{-\infty}\int_{\bR^d_+}  |D_1 u|^2 \, dx \, dt.
$$
 Hence, \eqref{eq_time_3} and an appropriate choice of $\varepsilon>0$ depending only on $\delta$ lead to \eqref{eq0925_1}.

2. When $T=\infty$,  we have
 $
\Re \left(\int^{\infty}_{-\infty}\int_{ \bR^d_+} \bar{u}^{\text{tr}} u_t \, dx \,dt\right) =0.
$
The rest is the same as step 1 with $T$ replaced by $\infty$.
\end{proof}

Proposition \ref{prop0925_1}  crucially supports the following theorem.

\begin{theorem}[Weighted $L_2$-theory with $\theta=d$ on a half space]
							\label{thm1011}
Let $T \in (-\infty, \infty]$, $\lambda \ge 0$. Then for any $u \in \fH_{2,d}^2((-\infty,T)\times\bR^d_+)$  satisfying the system
\begin{equation}
							\label{eq1002_1}
- u_t + A^{ij}(t) D_{ij} u - \lambda u = f
\end{equation}
in $(-\infty,T) \times \bR^d_+$, where $Mf \in \bL_{2,d}((-\infty,T)\times \bR^d_+)$, we have  
\begin{eqnarray}
							\label{eq1002_2}
 &&\lambda \| M u \|_{2,d}
+ \sqrt\lambda \| MDu \|_{2,d}+\|u\|_{\fH^2_{2,d}}
\,\,\le \,\, N \|M f\|_{2,d},
\end{eqnarray}
 where  $N$ depends only on  $d,d_1,\delta$,  $\| \cdot \|_{2,d} = \| \cdot \|_{\bL_{2,d}((-\infty,T) \times \bR^d_+)}= \| \cdot \|_{L_{2}((-\infty,T) \times \bR^d_+)}=\|\cdot\|_2$, and $\|u\|_{\fH^2_{2,d}}=\|u\|_{\fH^2_{2,d}((-\infty,T) \times \bR^d_+))}$. Moreover, for any $f$ satisfying $M f \in \bL_{2,d}((-\infty,T) \times \bR^d_+)$, there exists a unique solution $u \in \fH_{2,d}^2((-\infty,T) \times \bR^d_+)$ to the system \eqref{eq1002_1}.
\end{theorem}

\begin{proof}

First we prove  the a prior estimate (\ref{eq1002_2}) given that $u \in \fH_{2,d}^2((-\infty,T) \times \bR^d_+)$ satisfies the system \eqref{eq1002_1}.
For the argument below we may assume that $\lambda>0$. Then, since
$$
\lambda u = -u_t + A^{ij}(t) D_{ij}u - f,
$$
we have $\lambda M u \in L_2\left((-\infty,T) \times \bR^d_+\right)$.
Then, by the denseness results (see Theorem 1.19 and Remark 5.5 in \cite{MR1708104}), $u$ can be approximated by functions $u_n$ in $C_0^\infty\left((-\infty,T] \times \bR^d_+\right)$ which satisfy two limits
$$
\|u - u_n\|_{\fH_{2,d}^2} \to 0 \quad \text{and} \quad \|Mu-Mu_n\|_{2,d} \to 0
$$
as $n \to \infty$.
Moreover, by the interpolation inequality \cite[Lemma 3.3]{MR3318165}, we have
$$
\sqrt{\lambda}\|MD(u-u_n)\|_2 \leq N \lambda \|M(u-u_n)\|_2 + N \|M D^2(u-u_n)\|_2  \to 0.
$$
Due to this observation, we may just assume $u\in C^{\infty}_0((-\infty,T]\times \bR^d_+)$ and therefore  we get the estimate (\ref{eq1002_2}) from Proposition \ref{prop0925_1} and Lemma \ref{lem half}.

Thanks to the method of continuity, to prove the second assertion of the theorem for the unique solvability, we only need  the solvability of  the system $-u_t + \Delta u - \lambda u = f$, where $ \Delta u= [\Delta u^1 \cdots \Delta u^{d_1}]^{\text{tr}}$ which in turn follows from the solvability of the single equation $-v_t+\Delta v - \lambda v=g$ with the scalar valued functions $v$ and $g$. This is proved  in Theorem 3.5 of  \cite{MR3318165}. The theorem is proved.
\end{proof}

\section{Mean oscillation estimates of the first derivatives}

In this section we extend Theorem \ref{thm1011} to the case $p>1$ and also prepare key elements for the next section. We note that Proposition \ref{prop0925_1} was crucial for Theorem \ref{thm1011} and used the big advantage of $p=2$ to estimate $Du$, which in turn dominates $M^{-1}u$ via Hardy's inequality.  As we consider all $p>1$, we can no longer enjoy this. Instead, we will estimate the \emph{mean oscillation} of $Du$ and then estimate $Du$ in the frame of  sharp function  and maximal function theory.

To deal with the mean oscillation of the first derivatives away from and near the boundary, we first pose Lemmas \ref{lemma_MO_1} and \ref{lem boundary}. These are similar to Lemmas 4.2 and 4.3 in \cite{MR3318165}, the single equation results, which are based on unweighted $L_p$-estimates along with the standard localization and Sobolev embeddings.
Since the corresponding results for systems are available, for instance, in \cite{MR2771670}, we  only give  brief proofs.  The proofs are in the same spirit of  those in \cite{MR3318165}. For them we use the abbreviations $Q_r=Q_r(0,(0,{\bf{0}}))$,  $Q_r^+=Q^+_r(0,(0,{\bf{0}}))$; see Section \ref{sec_prelim} for the definitions of cylinders $Q_r(t,x)$, $Q^+_r(t,x)$.

\begin{lemma}[Interior H\"{o}lder estimate of $Du$]\label{lemma_MO_1}
Let $\lambda \ge 0$, $1< p \le q <\infty$, and $u \in \bW_p^{1,2}(Q_2)$ satisfy the system
$$
-u_t + A^{ij}(t) D_{ij} u - \lambda u = 0
$$
in $Q_2$.
Then $u$ belongs to $\bW_q^{1,2}(Q_1)$ and there exists a constant $N = N(d,d_1, \delta, p, q)$ such that
\begin{equation}
							\label{eq0301_01}
\|u\|_{\bW_q^{1,2}(Q_1)} \le N \|u\|_{\bL_p(Q_2)}.
\end{equation}
Moreover,  for the case $q > d+2$ we have
\begin{equation}\label{eq20210619_3}
\|Du\|_{C^{\alpha/2,\alpha}(Q_1)} \le N \| \sqrt{\lambda}\,|u| + |Du|\,\|_{\bL_p(Q_2)},
\end{equation}
where $\alpha = 1 - (d+2)/q \in (0,1)$ and $N=N(d,d_1,\delta,p,q)$.
\end{lemma}
\begin{proof}
This lemma is a system version of  \cite[Lemma 4.2]{MR3318165} and we do not see any obstacle when we follow the proof of it, which owes  its ideas to \cite{MR2304157}.  So here we just sketch the proof. 

First, we obtain \eqref{eq0301_01} for the case $q=p$ by a well known localization argument based on a sequence of increasing domains from $Q_1$ to $Q_2$ and the applications of unweighted $L_p$-estimate for systems like the first part of Proposition \ref{pro entire}. Then \eqref{eq0301_01} for any $q>p$ follows from a standard bootstrap argument. The inequality \eqref{eq20210619_3} with $\lambda=0$ can be handled by considering $u(t,x)-\int_{Q_2}u$ and then applying \eqref{eq0301_01} and the Sobolev embedding theorem. For the case $\lambda>0$ we rely on S. Agmon's idea of raising one more space dimension to use \eqref{eq0301_01}. While doing so, we do not forget to check that the enahnced system is still under control of our Legendre-Hadamard ellipticity condition \eqref{eq_intro_1}.
\end{proof}

Note that in the estimate \eqref{eq0301_01} the constant $N$ is independent of $\lambda (\ge 0)$.

\begin{lemma}[Boundary H\"{o}lder estimate of $Du$]
                                      \label{lem boundary}
Let $\lambda \ge 0$, $1<p \le q < \infty$, and $u \in \fH_{p,d}^2(Q_2^+)$ satisfy the system
$$
-u_t + A^{ij}(t) D_{ij} u - \lambda u = 0
$$
in $Q_2^+$.
Then $u$ belongs to $\bW_q^{1,2}(Q_1^+)$ and in fact there exists a constant $N = N(d,d_1, \delta,p,q)$ such that
\begin{equation}\label{eq20210619_1}
\|u\|_{\bW_q^{1,2}(Q_1^+)} \le N \|u\|_{\bL_p(Q_2^+)}.
\end{equation}
In particular,  for the case $q > d+2$ we have
\begin{equation}\label{eq20210619_2}
\|Du\|_{C^{\alpha/2,\alpha}(Q_1^+)} \le N \|u\|_{\bL_p(Q_2^+)},
\end{equation}
where $\alpha = 1 - (d+2)/q \in (0,1)$ and $N=N(d,d_1,\delta,p,q)$.
\end{lemma}

\begin{proof}
1. To proceed as in the (sketch of the) proof Lemma \ref{lemma_MO_1},  we will first  show $u\in \bW_p^{1,2}(Q_{3/2}^+)$. 

As argued in the proof of  \cite[Lemma 4.3]{MR3318165}, we may assume that $\lambda > 0$.
Since $u \in \fH_{p,d}^2(Q_2^+)$, we have
$$
M^{-1}u, \, Du \in \bL_p(Q_2^+),
$$
which in particular implies that $u, Du \in \bL_p(Q_2^+)$.
Consider an infinitely differentiable function  $\eta=\eta(t,x)$ defined in $\bR \times \bR^d$ such that $
0 \leq \eta \leq 1$, $\eta = 1$ on 
$$
Q_{3/2} = (-(3/2)^2,0) \times (-3/2,3/2) \times B_{3/2}'({\textbf{0}}),
$$
and $\operatorname{supp} (\eta) \subset (-4,4) \times (-2,2) \times B_2'({\textbf{0}})$.
Then $d_1\times 1$ matrix-valued function $\eta u$ satisfies the system
$$
-(\eta u)_t + A^{ij}(t)D_{ij}(\eta u) - \lambda (\eta u) = g
$$
in $(-\infty,0) \times \bR^d_+$, where $g := -\eta_t u + D_{ij} \eta A^{ij}(t)\,u + A^{ij}(t)(D_i\eta D_j u+D_j\eta D_i u)$. We  can extend $\eta u$ to $(-\infty,0) \times \bR^d_+$ with value zero outside $Q_2^+$ and doing so $\eta u$ still vanishes on $\{x_1=0\}$ as $u$ does. On the other hand, we notice that
$$
g \in \bL_p\left((-\infty,0) \times \bR^d_+\right).
$$
Hence, by Proposition \ref{pro entire} there exists a unique $w \in \bW_p^{1,2}\left((-\infty,0) \times \bR^d_+\right)$ satisfying $w(t,0,x') = 0$ and
$$
- w_t + A^{ij}(t)D_{ij} w - \lambda w = g
$$
in $(-\infty,0) \times \bR^d_+$. Moreover, since the support of $\eta$ is bounded, we have $M g \in \bL_p\left((-\infty,0) \times \bR^d_+\right)$. Hence, from Lemma \ref{lem half} with $\theta=d$, it follows that $w \in \fH_{p,d}^2\left((-\infty,0) \times \bR^d_+\right)$.
Noticing $\eta u \in \fH_{p,d}^2\left((-\infty,0) \times \bR^d_+\right)$ as $u \in \fH_{p,d}^2\left(Q_2^+\right)$ and by the uniqueness result of Theorem \ref{thm1011}, we have $w = \eta u$.
This means $u\in \bW_p^{1,2}(Q_{3/2}^+)$. 

2. Now the rest of the proof is a routine. As we  explained in the proof of Lemma \ref{lemma_MO_1}, we first obtain \eqref{eq20210619_1} for the case $q=p$ with $Q^+_{3/2}$ in place of $Q^+_{2}$ by a  localization argument based on a sequence of increasing domains from $Q^+_1$ to $Q^+_{3/2}$ and the applications of unweighted $L_p$-estimate for systems like the second part of Proposition \ref{pro entire}. Then again \eqref{eq20210619_1} for any $q>p$ follows from a standard bootstrap argument. The inequality \eqref{eq20210619_2} is just a Sobolev embedding.
\end{proof}

For vector-valued functions $u$, denote
$$
\left(u\right)_Q
= :\dashint_Q u(t,x) \, dx \, dt=\frac{1}{|Q|}\int_Q u(t,x) \, dx \, dt,
$$
where $Q \subset \bR \times \bR^d$ and $|Q|$ denotes the volume, Lebesgue measure,  of $Q$.
 
Below we abbreviate $Q_{\kappa r}^+(0,(y_1,{\bf 0}))$ by $Q_{\kappa r}^+(y_1)$. 

\begin{lemma}\label{lem0106_1}
Let  $\kappa \ge 32$,   $y_1\ge 0$, $\lambda \ge 0$, and $r>0$.
Assume that  $M f $ belongs to $\bL_2\left(Q_{\kappa r}^+(y_1)\right)$ and
$u \in \fH_{2,d}^2 \left(Q_{\kappa r}^+(y_1)\right)$ is a solution to the system
$$
-u_t + A^{ij}(t) D_{ij} u - \lambda u = f
$$
in $Q_{\kappa r}^+(y_1)$. Then we have the estimate 
\begin{eqnarray}
 \nonumber
 \left( \left| Du - \left(Du\right)_{Q_r^+(y_1)}\right|^2\right)^{1/2}_{Q_r^+(y_1)} &\le& N \kappa^{-1/2} \left( \sqrt{\lambda} \left(|u|^2\right)^{1/2}_{Q_{\kappa r}^+(y_1)} + \left(|Du|^2\right)^{1/2}_{Q_{\kappa r}^+(y_1)}\right)\\
&&+ N \kappa^{(d+2)/2 } \left( |M f|^2 \right)^{1/2}_{Q_{\kappa r}^+(y_1)}, \label{eq0120_71}
\end{eqnarray}
where  $ \left| Du - \left(Du\right)_{Q_r^+(y_1)}\right|^2$ denotes $\sum_{i=1}^d \left| D_iu - \left(D_iu\right)_{Q_r^+(y_1)}\right|^2 $  and $N = N(d,d_1, \delta, q) > 0$; in particular, $N$ is indendent of $f$, $u$, $y_1$, $\lambda$, $r$. 
\end{lemma}

\begin{proof}
1. By noticing  $M u \in \bL_2\left(Q_{\kappa r}^+(y_1)\right)$ and considering $-u_t + A^{ij}(t) D_{ij} u - \varepsilon u = f - \varepsilon u$ and letting $\varepsilon \searrow 0$, we can only  consider the case $\lambda > 0$.

 Moreover, we  only need to  prove the result for the special case  $r=\frac{8}{\kappa}$ ($\kappa r = 8$). 
In fact,  let $y_1\ge 0$, $\lambda>0$, $r>0$ be any  numbers. Then for any $f$, $u$ defined  on  $Q^+_{\kappa r}(y_1)$ and satisfying the given assumptions, we define 
$$
v(t,x)=u(\beta^2 t, \beta x), 
\quad g(t,x)=\beta^2 f(\beta^2 t, \beta x),
$$
where  $\beta:=\frac{\kappa r}{8}$. Then  $v,g$ are functions defined on $Q_{8}^+(y_1/\beta)$, $Mv$  is in $\bL_2\left(Q_{8}^+(y_1/\beta)\right)$, and $u$ is in  $\fH_{2,d}^2 \left(Q_{8}^+(y_1/\beta)\right)$. Moreover, $v$ is a solution to the system
\begin{equation*}
-v_t + A^{ij}(\beta^2t) D_{ij} v - \lambda\beta^2 v = g\label{eq_MO_2}
\end{equation*}
in $Q_{8}^+(y_1/\beta)$. Hence, if the lemma holds when $kr=8$, then  we have (\ref{eq0120_71}) with $v,g,y_1/\beta,\lambda\beta^2,r=8/\kappa$ in place of $u,f,y_1,\lambda,r$, respectively. On the other hand, a straightforward computation shows that
$$
\left( |M g|^2 \right)^{1/2}_{Q_{8}^+(y_1/\beta)}=\beta\left( |M f|^2 \right)^{1/2}_{Q_{\kappa r}^+(y_1)},\quad
\left(|Dv|^2\right)^{1/2}_{Q_{8}^+(y_1/\beta)}=\beta\left(|Du|^2\right)^{1/2}_{Q_{\kappa r}^+(y_1)},
$$
$$
\sqrt{\lambda\beta^2} \left(|v|^2\right)^{1/2}_{Q_{8}^+(y_1/\beta)} = \beta\sqrt{\lambda}\left(|u|^2\right)^{1/2}_{Q_{\kappa r}^+(y_1)},
$$
$$
 \left( \left| Dv - \left(Dv\right)_{Q_{8/\kappa}^+(y_1/\beta)}\right|^2\right)^{1/2}_{Q_{8/\kappa}^+(y_1/\beta)} =\beta \left( \left| Du - \left(Du\right)_{Q_r^+(y_1)}\right|^2\right)^{1/2}_{Q_r^+(y_1)},
$$
and we obtain (\ref{eq0120_71}) for general $r>0$.
Thus, the result of this lemma for the special $r=\frac{8}{\kappa}$ implies the result for  general $r>0$.  

2. Let us first consider the case $y_1 \in [0,1]$.
Since we assume $r = 8/\kappa \le 1/4$, we will keep the following in our mind:
\begin{multline*}
Q_r^+(y_1)=(-r^2,0)\times((y_1-r)\vee 0,y_1+r)\times B'_r({\textbf{0}})\\ \subset Q_2^+ \subset Q_4^+ \subset Q_{\kappa r}^+(y_1)=Q^+_8(y_1)
\end{multline*}
as $(0,y_1+r)\subset (0,2)$ and $(0,4)\subset(0,y_1+8)$, where $a\vee b:=\max\{a,b\}$. 
We note $MfI_{Q^+_4}\in \bL_{2,d}((-\infty,0)\times \bR^d_+)=\bL_2((-\infty,0)\times \bR^d_+)$. Hence, by Theorem \ref{thm1011}, there is a unique $w \in \fH_{2,d}^2((-\infty,0) \times \bR^d_+)$ satisfying the system
$$
-w_t + A^{ij}(t) D_{ij} w - \lambda w = f \textbf{1}_{Q_4^+}
$$
in $(-\infty,0) \times \bR^d_+$, where $\textbf{1}_{Q}$ denotes the indicator function on $Q$ and, in particular, we have
\begin{equation}
      \label{eq0120_5}
\|Dw\|_{\bL_{2,d}((-\infty,0) \times \bR^d_+)} \le N\|M f I_{Q_4^+} \|_{\bL_{2,d}((-\infty,0) \times \bR^d_+)} = N \|M f\|_{\bL_2(Q_4^+)},
\end{equation}
where $N=N(d,d_1,\delta)$. Then $v := u-w$ is in $\fH_{2,d}^2(Q_4^+)$ and satisfies
$$
-v_t+A^{ij}(t)D_{ij}v-\lambda v=0, \quad (t,x)\in Q^+_4.
$$
To obtain \eqref{eq0120_71} we utilize $w$ and $v$. 

By definitions we note that 
$$
\left( \left| Dw - \left(Dw\right)_{Q_r^+(y_1)}\right|^2\right)^{1/2}_{Q_r^+(y_1)}
  \leq
 2 \left( \left| Dw \right|^2\right)^{1/2}_{Q_r^+(y_1)}
$$
and for any $\alpha\in (0,1)$
\begin{eqnarray}
\left( \left| Dv - \left(Dv\right)_{Q_r^+(y_1)}\right|^2\right)^{1/2}_{Q_r^+(y_1)}
  \leq
 N r^{\alpha}[Dv]_{C^{\alpha/2,\alpha}(Q^+_2)},\label{eq_MO_1}
\end{eqnarray}
where $N=N(\alpha)$. On the other hand, by Lemma \ref{lem boundary} with $p=2$, $q$ satisfying $1-(d+2)/q = 1/2$, and a scaling argument as in step 1, we have
\begin{eqnarray}
[Dv]_{C^{1/4,1/2}(Q^+_2)}\leq N \|v\|_{\bL_2(Q^+_4)}
\leq N \|M^{-1}v\|_{\bL_{2}(Q^+_4)}\leq N \|Dv\|_{\bL_{2}(Q^+_4)}, \label{eq0120_4}
\end{eqnarray}
where the last inequality is due to Hardy's inequality and the last $N$ depends only on $d,d_1,q$. Combining (\ref{eq_MO_1}), (\ref{eq0120_4}), and  (\ref{eq0120_5}), we have
\begin{eqnarray*}
 && \left( \left| Du - \left(Du\right)_{Q_r^+(y_1)}\right|^2\right)^{1/2}_{Q_r^+(y_1)}\\
 &\leq& N \left( \left| Dv - \left(Dv\right)_{Q_r^+(y_1)}\right|^2\right)^{1/2}_{Q_r^+(y_1)}+
 N \left( \left| Dw \right|^2\right)^{1/2}_{Q_r^+(y_1)} \\
 &\leq& N r^{1/2} \left( \left| Dv \right|^2\right)^{1/2}_{Q_4^+} +N r^{-(d+2)/2} \left(\left|Mf \right|^2\right)^{1/2}_{Q_4^+}\\
&\leq& N r^{1/2} \left( \left| Du \right|^2\right)^{1/2}_{Q_4^+} +N r^{-(d+2)/2} \left(\left|Mf \right|^2\right)^{1/2}_{Q_4^+},
 \end{eqnarray*}
where $N=N(d,d_1,q)$.
 Since $\kappa r=8$ and $Q_4^+ \subset Q^+_{\kappa r}(y_1)=Q^+_{8}(y_1)$, we obtain (\ref{eq0120_71}).

 3. Let $y_1\in(1,\infty)$.  We again assume $r=8/{\kappa}\leq 1/4$. Due to $y_1>1$,   this time we have
 $$
 Q_r^+(y_1)=Q_r(y_1)\subset Q_{1/4}(y_1)\subset Q_{1/2}(y_1)\subset Q_{\kappa r}^+(y_1).
 $$
 As in  step 2, by Theorem \ref{thm1011} there is a unique solution $w\in \fH^2_{2,d}((-\infty,0)\times \bR^d_+)$ to the system
 $$
 -w_t+A^{ij}(t)D_{ij}w-\lambda w=f \textbf{1}_{Q_{1/2}(y_1)}
 $$
 and the estimate  (\ref{eq1002_2}) holds with $w$ and $f \textbf{1}_{Q_{1/2}(y_1)}$ in place of $u$ and $f$, respectively.  In particular, we have
\begin{multline*}
\lambda\|Mw\|_{\bL_{2,d}((-\infty,0)\times \bR^d_+)}+\|M^{-1}w\|_{\bL_{2,d}((-\infty,0)\times \bR^d_+)}+\|Dw\|_{\bL_{2,d}((-\infty,0)\times \bR^d_+)}
\\
\leq N \|Mf\|_{\bL_2(Q_{1/2}(y_1))},
\end{multline*}
where  $N=N(d,d_1,\delta)$. This estimate along with the inequality
 $$
 \sqrt{\lambda} \leq \lambda x_1+x_1^{-1},\quad x_1>0
 $$
 shows that
 \begin{equation}
          \label{eq0121_1}
           \|\sqrt{\lambda}|w|+|Dw|\|_{\bL_{2,d}((-\infty,0)\times \bR^d_+)}\leq N \|Mf\|_{\bL_2(Q_{1/2}(y_1))},
 \end{equation}
where $N=N(d,d_1,\delta)$. Then $v:=u-w\in \fH^2_{2,d}((-\infty,0)\times \bR^d_+)$ and satisfies
 $$
 -v_t+A^{ij}(t)D_{ij}v-\lambda v=0, \quad (t,x)\in Q_{1/2}(y_1).
 $$
 Applying Lemma \ref{lemma_MO_1} with $p=2$, a large $q$ satisfying $1-(d+2)/q = 1/2$, and scaling/translation arguments, we get
 \begin{eqnarray*}
 \nonumber
\left( \left| Dv - \left(Dv\right)_{Q_r^+(y_1)}\right|^2\right)^{1/2}_{Q_r^+(y_1)}
  &\leq&
 N r^{1/2}[Dv]_{C^{1/4,1/2}(Q_{1/4}(y_1))}\\
&\leq& N r^{1/2} \left((\sqrt{\lambda}|v|+|Dv|)^2\right)^{1/2}_{Q_{1/2}(y_1)},
\end{eqnarray*}
where $N=N(d,d_1,q)$. As in the last part of  step 2, we then have
\begin{eqnarray*}
 && \left( \left| Du - \left(Du\right)_{Q_r^+(y_1)}\right|^2\right)^{1/2}_{Q_r^+(y_1)}\\
 &\leq& N \left( \left| Dv - \left(Dv\right)_{Q_r^+(y_1)}\right|^2\right)^{1/2}_{Q_r^+(y_1)}+
 N \left( \left| Dw \right|^2\right)^{1/2}_{Q_r^+(y_1)} \\
 &\leq& N r^{1/2}  \left((\sqrt{\lambda}|v|+|Dv|)^2\right)^{1/2}_{Q_{1/2}(y_1)} +N r^{-(d+2)/2} \left(\left|Mf \right|^2\right)^{1/2}_{Q_{1/2}(y_1)}\\
&\leq& N r^{1/2}  \left((\sqrt{\lambda}|u|+|Du|)^2\right)^{1/2}_{Q_{1/2}(y_1)}+N r^{-(d+2)/2} \left(\left|Mf \right|^2\right)^{1/2}_{Q_{1/2}(y_1)},
 \end{eqnarray*}
where we used \eqref{eq0121_1} for the last inequality and  the last $N$ depends only on $d,d_1,q$.
 Since $\kappa r=8$ and $Q_{1/2}(y_1) \subset Q^+_{\kappa r}(y_1)=Q^+_{8}(y_1)$,   (\ref{eq0120_71}) follows again.
 \end{proof}

\begin{remark}
   \label{re0126_1}
For $T \in (-\infty,\infty]$, consider the collection of all parabolic cylinders in $(-\infty,T) \times \bR^{d}_+$:
$$
\cQ=\{Q^+ = Q^+_r(t,x): (t,x) \in (-\infty,T) \times \bR^{d}_+,\; r \in (0,\infty)\}.
$$
For given $p>1$ we call a scalar valued function $w$ on $(-\infty,T) \times \bR^d_+$ \emph{Muckenhoupt weight} or \emph{$A_p$ weight} and write $w \in A_p((-\infty,T) \times \bR^d_+)$ if $w$ is a non-negative function defined on $(-\infty,T) \times \bR^d_+$ and satisfies
$$
[w]_{A_p}:=
\sup\{(w)_{Q^+}\cdot (w^{-1/(p-1)})_{Q^+}^{p-1}:Q^+\in \cQ\}<\infty.
$$
As a trivial example, $w\equiv 1$ is a $A_p$ weight for any $p>1$ as $[w]_{A_p}=1$.
In this paper the following observation is very important and will be used in the next section:  if $p \in (1,\infty)$, $1< q < p$, and $\theta \in (d-1,d-1+p/q)$, then
$$
w=w(t,x)=x_1^{\theta-d} \in A_{p/q}((-\infty,\infty) \times \bR^{d}_+).
$$
Indeed, for any $t$, $x=(x_1,x') \in \bR^d_+$, and $r > 0$, we have
$$
(w)_{Q^+_r(t,x)}=\frac{1}{2r}\int_{(x_1-r)\vee 0}^{x_1+r} y_1^{\theta-d} \, dy_1.
$$
If $x_1 < 2r$, then the length of the interval $((x_1-r)\vee 0, \;\;x_1+r)$ is either $2r$ or $x_1+r$, and thus it is greater than  $r$. Hence, 
\begin{eqnarray*}
&&(w)_{Q^+_r(t,x)}\cdot (w^{-1/(p/q-1)})^{p/q-1}_{Q^+_r(t,x)}
\\
&=&\left(\dashint_{(x_1-r)\vee 0}^{x_1+r} y_1^{\theta-d} \, dy_1\right) \left( \dashint_{(x_1-r)\vee 0}^{x_1+r} \left(y_1^{\theta-d}\right)^{-1/(p/q-1)} \, dy_1\right)^{p/q-1}
\\
&\leq& \left(\frac{1}{r}\int_0^{x_1+r} y_1^{\theta-d} \, dy_1\right) \left( \frac{1}{r}\int_0^{x_1+r} \left(y_1^{\theta-d}\right)^{-1/(p/q-1)} \, dy_1\right)^{p/q-1}
\\
&\leq& \left( r^{\theta-d} \int_0^{\frac{x_1}{r} + 1} \tau^{\theta-d} \, d\tau \right)\left( r^{-\frac{\theta-d}{p/q-1}} \int_0^{\frac{x_1}{r}+1} \tau^{-\frac{\theta-d}{p/q-1}} \, d\tau \right)^{p/q-1}
\\
&\leq& \left( \int_0^3 \tau^{\theta-d} \, d\tau\right) \left(\int_0^3 \tau^{-\frac{\theta-d}{p/q-1}} \, d\tau \right)^{p/q-1},
\end{eqnarray*}
where the last quantity is finite since $\theta - d > -1$ and $-\frac{\theta-d}{p/q-1}> -1$.
If $x_1 \geq  2r$, then $(x_1-r)\vee 0=x_1-r$ and we have
$$
\left(\frac{1}{2r}\int_{x_1-r}^{x_1+r} y_1^{\theta-d} \, dy_1\right) \left( \frac{1}{2r}\int_{x_1-r}^{x_1+r} \left(y_1^{\theta-d}\right)^{-1/(p/q-1)} \, dy_1\right)^{p/q-1}
$$
$$
= 2^{-p/q} \left(\int_{\frac{x_1}{r}-1}^{\frac{x_1}{r}+1} \tau^{\theta-d} \, d\tau \right) \left( \int_{\frac{x_1}{r}-1}^{\frac{x_1}{r}+1} \tau^{-\frac{\theta-d}{p/q-1}} \, d\tau\right)^{p/q-1}
$$
$$
\leq 
\left\{
\begin{aligned}
2^{-p/q} 2 \left(\frac{x_1}{r}-1\right)^{\theta-d} \left( 2 \left(\frac{x_1}{r}+1\right)^{-\frac{\theta-d}{p/q-1}} \right)^{p/q-1}, \quad \text{if} \quad \theta -d \leq 0,
\\
2^{-p/q} 2 \left(\frac{x_1}{r}+1\right)^{\theta-d} \left( 2 \left(\frac{x_1}{r}-1\right)^{-\frac{\theta-d}{p/q-1}} \right)^{p/q-1}, \quad \text{if} \quad \theta - d > 0,
\end{aligned}
\right.
$$
where the last quantities are bounded by a constant independent of $x_1$ and $r$ since
$$
\left(\frac{x_1}{r}+1\right)\left(\frac{x_1}{r}-1\right)^{-1} \leq 3
$$
holds as long as $x_1 \ge 2r$.

Now, let us record two theorems which we will use in connection with $A_p$ weights we just mentioned. First, we consider the (scalar valued) maximal function of matrix-valued function $g$,
$$
\cM  g (t,x) := \sup\left\{ (|g|)_{Q^+}:Q^+\in\cQ\;\text{and }Q^+\text{contains }(t,x)  \right\}, \quad
(t,x) \in (-\infty,T) \times \bR^{d}_+.
$$
Then we have  the following Hardy-Littlewood maximal function theorem with $A_p$ weights (WHL):
$$
\| \cM g \|_{\bL_{p,w}} \leq N \|g\|_{\bL_{p,w}},
$$
where
$$
\|f\|_{\bL_{p,w}}^p = \int_{-\infty}^T \int_{\bR^d_+} |f(t,x)|^p w(t,x) \, dx \, dt
$$
and  $N=N(d,p,[w]_{A_p})$ i.e. independent of $g$.

We will also use the Fefferman-Stein theorem for sharp functions with $A_p$ weights (WFS). To state this theorem precisely, we define our sharp functions using a filtration we now describe. Consider the following series of partitions of $(-\infty,T) \times \bR^d_+$.
$$
\mathcal{P}_\ell : = \{Q^\ell = Q^\ell_{i_0,i_1,\ldots,i_d}: i_0, i_1,,\ldots, i_d \in \bZ, \, i_0 \leq 0, \, i_1 \ge 0\},
$$
where $\ell \in \bZ$ and $Q^\ell_{i_0,i_1,\ldots,i_d}$ is the intersection of $(-\infty,T) \times \bR^d_+$ with parabolic cubes
$$
[(i_0-1)2^{-2\ell}+T, i_0 2^{-2\ell}+T) \times [i_1 2^{-\ell}, (i_1+1)2^{-\ell}) \times \cdots \times [i_d 2^{-\ell}, (i_d+1)2^{-\ell}),
$$
when $T < \infty$.
If $T = \infty$, we replace $i_0 \leq 0$ by $i_0 \in \bZ$ and the time interval $[(i_0-1)2^{-2\ell} + T, i_0 2^{-2\ell}+T)$ by
$[(i_0-1)2^{-2\ell}, i_0 2^{-2\ell})$. As $\ell$ increases, $\mathcal{P}_{\ell}$ becomes finer. We call $\mathcal{P}:=\bigcup_{\ell\in\bZ}\mathcal{P}_{\ell}$ a filtration of $(-\infty,T) \times \bR^d_+$. 
Then, we define the (scalar valued) sharp function of matrix-valued function $g$,
\begin{multline*}
g^{\#}_{\operatorname{dy}}(t,x): = \sup\left\{(|g -
(g)_{Q^\ell}|)_{Q^{\ell}} :Q^\ell \in\mathcal{P}\text{ and }Q^\ell\text{contains }(t,x) \right\},\\ \; (t,x) \in (-\infty,T) \times \bR^d_+.
\end{multline*}
The Fefferman-Stein theorem for sharp functions with $A_p$ weights (see, for instance, \cite[Theorems 2.3 and 2.4]{MR3812104}) states that
$$
\| g \|_{\bL_{p,w}} \leq N \|g^{\#}_{\operatorname{dy}}\|_{\bL_{p,w}}
$$
for $w \in A_p\left((-\infty,T) \times \bR^d\right)$, 
where $N = N(d,p,[w]_p)$.

\end{remark}

The following theorem extends Theorem {\ref{thm1011}} and considers all $p>1$. 
\begin{theorem}
[Weighted $L_p$-theory with $\theta=d$ on a half space]
							\label{thm_MO_1}
Let $T \in (-\infty, \infty]$, $\lambda \ge 0$, and $p\in(1,\infty)$.  Then for any $u \in \fH_{p,d}^2((-\infty,T)\times\bR^d_+)$  satisfying the system
\begin{equation}
							\label{eq_MO_3}
- u_t + A^{ij}(t) D_{ij} u - \lambda u = f
\end{equation}
in $(-\infty,T) \times \bR^d_+$ with $Mf \in \bL_{p,d}((-\infty,T)\times \bR^d_+)$, we have 
\begin{eqnarray}
							\label{eq_MO_4}
\lambda \| M u \|_{p,d}
+ \sqrt\lambda \| MDu \|_{p,d}+\|u\|_{\fH^2_{p,d}}\leq N \|M f\|_{p,d},
\end{eqnarray}
where  $N$ depends only on  $d,d_1,\delta,p$, $\| \cdot \|_{p,d} = \| \cdot \|_{\bL_{p,d}((-\infty,T) \times \bR^d_+)}$, and $\|\cdot\|_{\fH^2_{p,\theta}}=\|\cdot\|_{\fH^2_{p,\theta}((-\infty,T)\times \bR^d_+)}$. Moreover, for any $f$ satisfying $M f \in \bL_{p,d}((-\infty,T) \times \bR^d_+)$, there exists a unique solution $u \in \fH_{p,d}^2((-\infty,T) \times \bR^d_+)$ to the system \eqref{eq_MO_3}.
\end{theorem}
\begin{proof}
Due to the method of continuity and  the corresponding theory of the Laplacian case in  e.g. Theorem 3.5 in \cite{MR3318165},  we only prove the a priori estimate (\ref{eq_MO_4}).

1. Let $p>2$. Take any $\kappa\geq 32$.  Then using Lemma \ref{lem0106_1} with a simple translation argument, we have
\begin{eqnarray}
 \nonumber
 \left( \left| Du - \left(Du\right)_{Q^+_r(s,y)}\right|^2\right)^{1/2}_{Q^+_r(s,y)} 
&\le& N \kappa^{-1/2} \left( \sqrt{\lambda} \left(|u|^2\right)^{1/2}_{Q^+_{\kappa r(s,y)} }+ \left(|Du|^2\right)^{1/2}_{Q^+_{\kappa r(s,y)}}\right)\\
&&+ N \kappa^{(d+2)/2 } \left( |M f|^2 \right)^{1/2}_{Q^+_{\kappa r(s,y)}} \label{eq0120_7}
\end{eqnarray}
 for any $(s,y)\in (-\infty,T)\times\bR^d_+$ and $r>0$, where $N=N(d,d_1,\delta)$.

Now, fix any $(t,x) \in (-\infty,T) \times \bR^d$ for a moment. For each $\ell\in \bZ$ we consider the cube $Q^\ell \in \mathcal{P}_\ell$ containing $(t,x)$ and
find $Q_r^+(s,y)$, $(s,y) \in (-\infty,T) \times \bR^d_+$ with the smallest $r>0$ such that
$Q^\ell \subset Q_r^+(s,y)$ and
$$
\left( |Du - (Du)_{Q^\ell}|^2 \right)^{1/2}_{Q^\ell} \leq N\left(|Du - (Du)_{Q_r^+(s,y)}|^2 \right)^{1/2}_{Q_r^+(s,y)},
$$ 
where $N$ depends only on the ratio of the measures $\frac{|Q_r^+(s,y)|}{|Q^\ell|}$ and hence $N=N(d)$.

From this,  (\ref{eq0120_7}),  Jensen's inequality, and the definitions of sharp functions and maximal functions in Remark \ref{re0126_1},  we obtain
\begin{eqnarray}
(Du)^{\#}_{\operatorname{dy}}(t,x)&\le& N \kappa^{-1/2}\left(\sqrt{\lambda}\cM^{1/2}(|u|^2)(t,x)+\cM^{1/2}(|Du|^2))(t,x)\right)\nonumber\\
&&+N\kappa^{(d+2)/2}\cM^{1/2}(|Mf|^2))(t,x).\label{eq20210621_1}
\end{eqnarray}

The estimate \eqref{eq20210621_1} holds for any fixed $(t,x)\in(-\infty,T)\times\bR^d_+$.  Hence, we have
\begin{eqnarray*}
\|(Du)^{\#}_{\operatorname{dy}}\|^p_{p,d}
&\leq& N \kappa^{-p/2}\left((\sqrt{\lambda})^p\| \cM( |u|^2)\|^{p/2}_{p,d}+\|\cM(|Du|^2)\|^{p/2}_{p/2,d}\right)\\
            &&+N\kappa^{p(d+2)/2}\|\cM(|Mf|^2)\|^{p/2}_{p/2,d},
\end{eqnarray*}
where $N=N(d,d_1,\delta,p)$.  Noting $p/2>1$ in this step and applying WFS  and WHL in Remark $\ref{re0126_1}$ with $w \equiv 1$, which  we usually call FS theorem and HL theorem,
we have
\begin{eqnarray*}
\|Du\|^p_{p,d}
&\leq& N \kappa^{-p/2}\left((\sqrt{\lambda})^p\| |u|^2\|^{p/2}_{p/2,d}+\||Du|^2\|^{p/2}_{p/2,d}\right)+N\kappa^{p(d+2)/2}\||Mf|^2\|^{p/2}_{p/2,d}\\
&=& N \kappa^{-p/2}\left((\sqrt{\lambda})^p\| u\|^{p}_{p,d}+\|Du\|^{p}_{p,d}\right)+N\kappa^{p(d+2)/2}\|Mf\|^{p}_{p,d},
\end{eqnarray*}
and therefore
\begin{eqnarray*}
\|Du\|_{p,d}
&\le & N \kappa^{-1/2}\left(\|\sqrt{\lambda} u\|_{p,d}+\|Du\|_{p,d}\right)+N\kappa^{(d+2)/2}\|Mf\|_{p,d}\\
&\le& N \kappa^{-1/2}\left(\lambda\| Mu\|_{p,d}+\| M^{-1}u\|_{p,d}+\|Du\|_{p,d}\right)+N\kappa^{(d+2)/2}\|Mf\|_{p,d},
\end{eqnarray*}
where we used $\sqrt{\lambda}\le \lambda x_1+1/x_1$, $x_1>0$ for the second inequality.  
Then Lemma \ref{lem half} with $\theta =d$ and Hardy's inequality give
\begin{eqnarray*}
&&\lambda \|Mu\|_{p,d}+ \sqrt{\lambda} \|MDu\|_{p,d}+ \|u\|_{\fH^2_{p,d}} \nonumber\\
&\leq&
N\kappa^{-1/2}\left(\lambda\| Mu\|_{p,d}+\| M^{-1}u\|_{p,d}+\|Du\|_{p,d}\right)+N\kappa^{(d+2)/2}\|Mf\|_{p,d}+N\|Mf\|_{p,d},
\end{eqnarray*}
and an  appropriate choice of $\kappa\ge 32$ leads us to (\ref{eq_MO_4}).

2. Let $1<p<2$.  We use a duality argument with step 1.
Again it suffices to prove the a priori estimate \eqref{eq_MO_4}.
Furthermore, thanks to Lemma \ref{lem half}, we only need to prove that
\begin{equation}
							\label{eq0305_01}
\|M^{-1}u\|_p \leq 
N \|M f\|_p,
\end{equation}
where $\|\cdot\|_p = \|\cdot\|_{\bL_p((-\infty,T) \times \bR^d_+)}$. To prove this, we use the fact (see e.g.  \cite[Theorem 2.3]{MR1837532}) that $\bL_{p,d-p}((-\infty,T) \times \bR^d_+)$ is the dual space of $\bL_{q,d+p}((-\infty,T) \times \bR^d_+)$, where $1/p+1/q=1$ with $q>2$ now. 

Let $g \in \bL_{q,d+p}((-\infty,T) \times \bR^d_+)$, that is, $Mg \in \bL_{q,d}((-\infty,T) \times \bR^d_+)$.
Then, using the above result  applied with $A^{ij}(-t)$ and $q>2$,   we find that  there exists unique $v \in \fH_{q,d}^2(\bR \times \bR^d_+)$ satisfying
$$
v_t + A^{ij}(t) D_{ij} v - \lambda v = g I_{t \in (-\infty,T)}
$$
in $\bR \times \bR^d_+$. In particular, $v(t,x) = 0$ for $t \ge T$ in case $T < \infty$. This is because both $0$ and $\bar{u}(t,x):=v(-t,x)$ satisfy the system
$$
w_t=A^{ij}(-t)D_{ij}w-\lambda w
$$
on $(-\infty,-T)\times \bR^d_+$.
%
Thus  we have
\begin{eqnarray*}
\int^T_{-\infty}\int_{\bR^d_+} u^{\text{tr}}\; g\; \, dx \, dt 
&=& \int^T_{-\infty}\int_{\bR^d_+} u^{\text{tr}}\; \left(v_t+A^{ij}(t) D_{ij} v - \lambda v \right) \, dx \, dt\\
&=&\int^T_{-\infty}\int_{\bR^d_+} \left( - u_t + A^{ij}(t) D_{ij} u - \lambda u \right) ^{\text{tr}}v \, dx \, dt\\
&=&\int^T_{-\infty}\int_{\bR^d_+}f^{\text{tr}}\;  v \;\, dx \,dt\\
&=&\int^T_{-\infty}\int_{\bR^d_+}(x_1f)^{\text{tr}}\;  (x^{-1}v) \;\, dx \,dt\\
&\leq& \|M f\|_p \|M^{-1}v\|_q \leq N \|M f\|_p\|M g\|_q,
\end{eqnarray*}
where the last inequality holds by step 1. This shows that $\|M^{-1}u\|_p$ can not exceed $N\|M f\|_p$  i.e. \eqref{eq0305_01}.

3. Finally, Theorem \ref{thm1011} takes care of the case $p=2$.
\end{proof}

Based on Theorem \ref{thm_MO_1}, we build the following lemma, which  is an $L_p$- counterpart of Lemma \ref{lem0106_1}.
\begin{lemma}[Mean oscillation of $Du$ on a  half space]
							\label{lem0126_5}
Let $p>1$, $\lambda \ge 0$, $r > 0$, $\kappa \ge 32$, and $y_1 \ge 0$. Assume that $Mf \in \bL_p\left(Q_{\kappa r}^+(y_1)\right)$ and 
let $u \in \fH_{p,d}^2 \left(Q_{\kappa r}^+(y_1)\right)$ be a solution to the system
$$
-u_t + A^{ij}(t) D_{ij} u - \lambda u = f
$$
in $Q_{\kappa r}^+(y_1)$.
Then we have
\begin{multline*}
							\label{eq0126_7}
\left( \left| Du - \left(Du\right)_{Q_r^+(y_1)}\right|^p\right)^{1/p}_{Q_r^+(y_1)} \leq N \kappa^{(d+2)/p} \left( |Mf|^p \right)^{1/p}_{Q_{\kappa r}^+(y_1)}
\\
+ N \kappa^{-1/2} \left( \sqrt{\lambda} \left(|u|^p\right)^{1/p}_{Q_{\kappa r}^+(y_1)} + \left(|Du|^p\right)^{1/p}_{Q_{\kappa r}^+(y_1)}\right), 
\end{multline*}
where  $N = N(d,d_1,\delta, p) > 0$.
\end{lemma}

\begin{proof}
The proof repeats the proof of Lemma \ref{lem0106_1} word for word. The only difference is that we use  Theorem \ref{thm_MO_1} ($L_p$-estimate)  in place of Theorem \ref{thm1011} ($L_2$-estimate).
\end{proof}

\section{Proof of Theorem \ref{thm_main_1}}

We first recall  Assumption \textbf{A}($\rho,\varepsilon$), which is assumed in Theorem \ref{thm_main_1}.

\begin{lemma}
                                \label{lemma_main_1}
Let $T \in (-\infty, \infty]$, $\lambda\ge 0$, $p\in(1,\infty)$, $\theta\in(d-1,d-1+p)$,  $\rho\in (1/2,1)$, and   Assumption \textbf{A}$(\rho,\varepsilon)$ hold. Then there exists a positive constant $\varepsilon_0=\varepsilon_0(d,d_1,\delta,p,\theta)$ such that if  $\varepsilon\in (0,\varepsilon_0]$  and $u \in \fH_{p,\theta}^2((-\infty,T)\times \bR^d_+)$  satisfy
\begin{equation*}			
- u_t + A^{ij}(t,x)D_{ij}u -\lambda u = f
\end{equation*}
in $(-\infty,T) \times \bR^d_+$, where $Mf \in \bL_{p,\theta}((-\infty,T)\times \bR^d_+)$,  then 
\begin{equation*}				
\lambda \| M u \|_{p,\theta}
+ \|u \|_{\fH^2_{p,\theta}}
\le N \|Mf\|_{p,\theta}+N\| D u \|_{p,\theta},
\end{equation*}
 where $\|\cdot\|_{p,\theta}=\|\cdot\|_{\bL_{p,\theta}((-\infty,T)\times \bR^d)},  \|\cdot\|_{\fH^2_{p,\theta}}= \|\cdot\|_{\fH^2_{p,\theta}((-\infty,T)\times \bR^d)}$, and $N = N(d, d_1,\delta, p,\theta)$.
\end{lemma}
\begin{proof}
To prove this lemma we follow the proof of \cite[Lemma 5.1]{MR3318165}, the result of single equations, almost word for word.   Doing so, one notices that the regularity condition on $A^{ij}$s in this paper is a bit different from that in \cite{MR3318165},  however, we see that the mean oscillations with respect to the spatial variables on $B_R(x)$, $R \in (0,1/2]$, of the coefficients
$$
A_r^{ij}(\cdot,\cdot):= A^{ij}(\cdot/r^2,\cdot/r)
$$
can be made sufficiently small under  our Assumption \textbf{A}($\rho$, $\varepsilon$) when $x_1 \in (1,4)$ and we are safe to proceed.

Carrying out the proof, we make sure to apply the result of systems from \cite{MR2771670} at  the very step which corresponds to the one at which the result of single equations is applied in the proof of \cite[Lemma 5.1]{MR3318165}.\end{proof}
Comparing Lemma \ref{lemma_main_1} and our main result, Theorem \ref{thm_main_1},  the domination of $\| D u \|_{p,\theta}$ by  $\| Mf \|_{p,\theta}$ is crucial. In fact, we have carefully worked out for this throughout the paper. Yet, we need one more step, Proposition \ref{proposition_main_1}, in which we import the Muckenhoupt weight we mentioned and prepared in Remark \ref{re0126_1}. 

To deliver the proof of Proposition  \ref{proposition_main_1} effectively, we elaborate the following lemma in advance.
\begin{lemma}
     \label{lem0201_1}
Let $q\in (1,\infty)$,  $\theta\in\bR$, $\beta\in(1,\infty)$, and $\beta'=\frac{\beta}{\beta-1}$, the H\"{o}lder conjugate of $\beta$.  Let  $h > 0$, $\rho \in (1/2,1)$, $R \in (0,\rho h)$, $\kappa \ge 32$  and let
$$
u \in \fH_{\beta q,\theta}^2 (\bR \times \bR^d_+)
$$
be compactly supported on $Q_R(h)=Q_R(0,(h,{\bf{0}}))=Q^+_R(0,(h,{\bf{0}}))$.
Then under Assumption \textbf{A}$(\rho,\varepsilon)$ with any given $\varepsilon>0$,
for any  $ (s,y) \in\bR \times  \overline{\bR^d_+}$ and $r > 0$ we have the estimate
\begin{eqnarray}
&&\left( \left| Du - \left(Du\right)_{Q_r^+(s,y)}\right|^q\right)^{1/q}_{Q_r^+(s,y)}\nonumber\\
 &\le& N_0 \kappa^{-1/2} \left( \sqrt{\lambda} \left(|u|^q\right)^{1/q}_{Q_{\kappa r}^+(s,y)} + \left(|Du|^q\right)^{1/q}_{Q_{\kappa r}^+(s,y)}\right)\nonumber\\
&&+ N_1 \kappa^{(d + 2)/q}  \varepsilon^{1/{(\beta' q)}} \left( |MD^2 u|^{\beta q} \right)^{1/(\beta q)}_{Q_{\kappa r}^+(s,y)}
+ N_0  \kappa^{(d+ 2)/q}  \left( |M f|^q \right)^{1/q}_{Q_{\kappa r}^+(s,y)},\nonumber\\
\label{eq20210622_1}
\end{eqnarray}
where  $N_0=N_0(d,d_1,\delta,q)$, $N_1=N_1(d,d_1,\delta,q,\beta,\rho)$, and
$$
f: = -u_t + A^{ij}(t,x)D_{ij}u - \lambda u
$$
in  $Q_{\kappa r}^+(s,y)$.
\end{lemma}

\begin{proof}
1. Since $u$ is supported on 
$$
Q_R(h) = (-R^2,0) \times (h-R,h+R) \times B_R'({\bf{0}})
$$
with $h - R > 0$, $u$ is supported on a compact set strictly away from the boundary of $\bR^d_+$ and hence, for any $(s,y) \in \bR \times \overline{\bR^d_+}$ and $r>0$, we have
\begin{equation}
    \label{eqn 5.18.1}
u \in \fH_{\beta q, d}^2(Q_{\kappa r}^+(s,y)) \cap \fH_{q, d}^2(Q_{\kappa r}^+(s,y)).
\end{equation}
On the other hand, by scaling argument we only need to show \eqref{eq20210622_1}  for the case $h = 1$.

2. Obviously, we may assume that $Q_r^+(s,y) \cap Q_R(1) \neq \emptyset$, which in particular means that the interval $((y_1-r)\vee 0,y_1+r)$ intersects with the interval $(1-R,1+R)$ and hence 
\begin{equation}
							\label{eq0309_03}
(1-R-r) < y_1 < 1+R+r,\quad y_1\ge 0.
\end{equation}
Note that $1-R-r$ can be negative with large $r>0$. 

We work with $Q^+_{\kappa r}(s,y)$ and $\textbf{A}(\rho,\varepsilon)$. To do so, 
we first observe the following two cases, depending on the size of $\kappa r$.

Case 1: $\kappa r \leq \rho(1-R-r)$.  In this case   we are forced to have $0<1-R-r$ and $r$ can not be arbitrarily large.   Along with it, \eqref{eq0309_03}  gives
$$
y_1 > \kappa r/\rho>\kappa r.
$$

Case 2: $\kappa r > \rho (1-R-r)$. In this case $r$ and hence $\kappa r$ can not be arbitrarily small. Indeed, this along with $\rho< 1 < \kappa$ shows that
\begin{equation}
							\label{eq0309_01}
\kappa r > (\rho+\kappa)r/2 > \rho r/2 + (1-R-r)\rho/2 = (1-R)\rho/2 > R(1-\rho)/2,
\end{equation}
where the last inequality follows $R < \rho$. If $\kappa r$ is large, $Q_{\kappa r}(s,y)$ may not be contained in $\bR\times\bR^d_+$.

Below in step 3, we will use Assumption \textbf{A}$(\rho,\varepsilon)$  for $A^{ij}$s and
in this condition the mean oscillation of $A^{ij}$ works only with  the parabolic cylinders $Q$ \emph{contained} in $\bR\times\bR^d_+$. Connected to this concern, 
we set $Q:=Q_{\kappa r}(s,y)=Q^+_{\kappa r}(s,y)$ in Case 1 and set $Q:= Q_R(1)$ in Case 2, noting that the support of $u$ is in $Q_R(1)$ and in Case 2
\begin{equation}
							\label{eq0309_02}
|Q_R(1)| = N(d) R^{d+2} \leq N(d,\rho) (\kappa r)^{d+2} \leq N |Q_{\kappa r}^+(s,y)|
\end{equation}
holds by \eqref{eq0309_01}.

3. Now, we  set
$$
\bar{A}^{ij}(t) = \dashint_B A^{ij}(t,z) \, dz,\quad i, j=1,\ldots,d,
$$
where $B$ is either $B_{\kappa r}(y)=(y_1-r,y_1+r)\times B'_{\kappa r}(y')$ or $B_R(1,{\bf{0}})=(1-r,1+r)\times B' _R({\bf{0}})$ depending on Case 1 or Case 2 in step 2, respectively.  We note that $\bar{A}^{ij}(t)$s depend only on $t$ and satisfy the Legendre-Hadamard condition \eqref{eq_intro_1} and the inequality $(|\bar{A}^{ij}-A^{ij}|)_Q\le \varepsilon$ holds  for all $i,j$ in both cases of $Q$ in step 2 by Assumption \textbf{A}$(\rho,\varepsilon)$ and the related definitions therein.

We then have the system
$$
-u_t + \bar{A}^{ij}(t) D_{ij} u - \lambda u = F
$$
on any chosen $Q_{\kappa r}^+(s,y)$,
where the $d_1\times 1$ matrix valued function $F$ is defined by
$$
F (t,x)= \left(\bar{A}^{ij}(t) - A^{ij}(t,x)\right) D_{ij}u(t,x) + f(t,x).
$$
By \eqref{eqn 5.18.1}, we have
$u \in \fH_{q,d}^2(Q_{\kappa r}^+(s,y))$ and $M F \in L_q(Q_{\kappa r}^+(s,y))$.
Then by Lemma \ref{lem0126_5} with  $q$  in place of $p$  along with a translation argument,  we have
\begin{eqnarray}
\nonumber
\left( \left| Du - \left(Du\right)_{Q_r^+(s,y)}\right|^q\right)^{1/q}_{Q_r^+(s,y)} &\le& N \kappa^{-1/2} \left( \sqrt{\lambda} \left(|u|^q\right)^{1/q}_{Q_{\kappa r}^+(s,y)} + \left(|Du|^q\right)^{1/q}_{Q_{\kappa r}^+(s,y)}\right)\\
&&+ N \kappa^{(d+ 2)/q} \left( |MF|^q \right)^{1/q}_{Q_{\kappa r}^+(s,y)}, \label{eqn 5.18.2}
\end{eqnarray}
where $N=N(d,d_1,\delta,q)$.
Meanwhile, by the definition of $F$, triangle inequality, H\"{o}lder inequality with \eqref{eqn 5.18.1}, the boundedness condition \eqref{eq_intro_3} with the fact $\beta'q>1$, and the observation \eqref{eq0309_02}, we have 
\begin{eqnarray*}
&&\left( |M F|^q \right)^{1/q}_{Q_{\kappa r}^+(s,y)}\nonumber\\
&\le& \sum_{i,j}\left(|\bar{A}^{ij}-A^{ij}|^{\beta'q}I_{Q}\right)^{1/(\beta'q)}_{Q^+_{\kappa r}(s,y)} \left(|MD^2u|^{\beta q}\right)^{1/(\beta q)}_{Q^+_{\kappa r}(s,y)}+\left( |Mf|^q \right)^{1/q}_{Q_{\kappa r}^+(s,y)}\nonumber\\
&\le& N'\sum_{i,j}\left(|\bar{A}^{ij}-A^{ij}|\right)^{1/(\beta'q)}_{Q} \left(|MD^2u|^{\beta q}\right)^{1/(\beta q)}_{Q^+_{\kappa r}(s,y)}+\left( |Mf|^q \right)^{1/q}_{Q_{\kappa r}^+(s,y)},\label{eq_main_6}
\end{eqnarray*}
where $N'$ depends only on $d_1,\delta,q,\beta$.  Then, by Assumption \textbf{A}$(\rho,\varepsilon)$  we obtain
\begin{eqnarray*}
\left( |M F|^q \right)^{1/q}_{Q_{\kappa r}^+(s,y)}\leq N^{''}\varepsilon^{1/{(\beta' q)}} \left( |MD^2 u|^{\beta q} \right)^{1/(\beta q)}_{Q_{\kappa r}^+(s,y)} +  \left( |Mf|^q \right)^{1/q}_{Q_{\kappa r}^+(s,y)},
\end{eqnarray*}
where $N^{''}= N_0(d,d_1,\delta,q,\beta,\rho)$.  This with \eqref{eqn 5.18.2} proves the lemma.  
\end{proof}

\begin{proposition}
                                \label{proposition_main_1}
Let $T \in (-\infty, \infty]$, $\lambda\ge 0$, $p\in(1,\infty)$, and $\theta\in(d-1,d-1+p)$. Also, let $h>0$, $\rho\in (1/2,1)$, $\varepsilon \in (0,\varepsilon_0]$, where $\varepsilon_0$ is taken from Lemma \ref{lemma_main_1}, and $R\in (0,\rho h)$. 

Let $u\in \fH^2_{p,\theta}((-\infty,T)\times\bR^d_+)$ be compactly supported on $Q_R(h)$ and let us set $f:=-u_t+A^{ij}(t,x)D_{ij}u-\lambda u$.
Then under Assumption \textbf{A}$(\rho,\varepsilon)$, we have
\begin{equation}
							\label{eq_main_3}
\lambda \| M u \|_{p,\theta}+\sqrt{\lambda} \| M Du \|_{p,\theta}
+ \|u \|_{\fH^2_{p,\theta}}
\le N_0 \|Mf\|_{p,\theta}  + N_1\varepsilon^{1/(\beta'q)} \|MD^2u\|_{p,\theta},
\end{equation}
where $\|\cdot\|_{p,\theta}=\|\cdot\|_{\bL_{p,\theta}((-\infty,T)\times \bR^d)}$, $\|\cdot\|_{\fH^2_{p,\theta}}= \|\cdot\|_{\fH^2_{p,\theta}((-\infty,T)\times \bR^d)}$, $N_0 = N_0(d,d_1,\delta,p,\theta)$, $N_1 = N_1(d, d_1,\delta, p,\theta,\rho)$,
and constants $\beta',q$ are positive numbers determined by $p$ and $\theta$.
\end{proposition}
\begin{proof}
For  $p\in(1,\infty)$, $\theta\in(d-1,d-1+p)$, we choose and fix $q$, $\beta \in (1,\infty)$ satisfying
$$
q\in (1,p), \quad q\beta<p, \quad  \theta\in (d-1, d-1+p/{\beta q}).
$$
Then we note $u\in \fH^2_{\beta q,\theta}((-\infty,T)\times\bR^d_+)$ since the support of $u$ is bounded.  Hence we can use Lemma \ref{lem0201_1}.  
Following  the arguments in the proof of Theorem \ref{thm_MO_1} with the help of Lemma \ref{lem0201_1},   we obtain that, for any $\kappa\ge 32$  and $(t,x) \in (-\infty,T) \times \bR^d_+$, we have the following sharp function and maximal function(s) relation
\begin{eqnarray*}
(Du)_{\operatorname{dy}}^{\#}(t,x)& \le&  N_0 \kappa^{-1/2} \left(\sqrt{\lambda} \cM^{1/q}\left( |u|^q \right)(t,x)+ \cM^{1/q}\left( |Du|^q \right)(t,x)\right)\\
&&+N_1\kappa^{(d+ 2)/q}  \varepsilon^{1/{(q\beta')}}\cM^{1/(q\beta)}\left( |MD^2u|^{q\beta} \right)(t,x)\\
&&+N_0 \kappa^{(d+ 2)/q} \cM^{1/q}\left( |Mf|^q \right)(t,x),
\end{eqnarray*}
where $\beta'=\frac{\beta}{\beta-1}$. Now, Remark \ref{re0126_1} comes into play.
As noted in the remark, $w=x_1^{\theta-d}$ is a Muckenhoupt weight and belongs to $A_{p/(\beta q)} \subset A_{p/q}$.
Hence, by Fefferman-Stein theorem (WFS) and Hardy-Littlewood theorem (WHL) for this weight $w$ we have
\begin{eqnarray*}
\|Du\|_{p,\theta} &\leq& N_0 \kappa^{-1/2} \left(\sqrt{\lambda}\|u\|_{p,\theta}+\|Du\|_{p,\theta}\right) \\
&&+N_1\kappa^{(d +\theta+ 2)/q}  \varepsilon^{1/{(q\beta')}}\|MD^2u\|_{p,\theta}+N_0
 \kappa^{(d+\theta + 2)/q}\|Mf\|_{p,\theta}.
 \end{eqnarray*}
Further, by the relation $\sqrt{\lambda} \leq \lambda x_1 + x_1^{-1}$ for $x_1 > 0$ mentioned earlier, we also have
 $$
 \sqrt{\lambda}\|u\|_{p,\theta}\leq N(p)\left( \lambda \|Mu\|_{p,\theta}+\|M^{-1}u\|_{p,\theta}\right).
 $$
Then along with these estimates, Lemma {\ref{lemma_main_1}}, an appropriate choice of  sufficiently large $\kappa\ge 32$, and the interpolation \eqref{eq20210624_1} we arrive at the estimate (\ref{eq_main_3}).
\end{proof}

\begin{remark}
The result of Proposition  \ref{proposition_main_1} is invariant under the translation of $Q_R(h)$ as long as the compact support of $u$ is contained in the translated cylinder.
\end{remark}
We are ready to wrap this paper up. 
\\  \\
{\bf{Proof of Theorem \ref{thm_main_1}}}
\begin{proof}
1. Due to the method of continuity and the corresponding theory of the Laplacian case in \cite[Theorem 3.5]{MR3318165}, it suffices to show the a priori estimate (\ref{eq_main_1}). 

 2. In this step,  we assume that $B^i(t,x)$s and $C(t,x)$  are zero matrices for all $t,x$.  

From Lemma 5.6 in \cite{MR2990037} (or Lemma 3.3 in \cite{MR2111792}) we bring in and prepare the following: for any constant $\varepsilon_2>0$  there is a constant $\rho = \rho(\varepsilon_2) \in (1/2,1)$ and non-negative (scalar) functions  $\eta_k=\eta_k(t,x) \in C_0^\infty(\bR\times\bR^{d}_+)$, $k=1,2,\ldots$ satisfying
\begin{equation}
							\label{eq_main_4}
\sum_k \eta_k^p \ge 1,
\quad
\sum_k \eta_k \le N(d),
\quad
\sum_k \left(M |D\eta_k| + M^2 |D^2 \eta_k| + M^2 |(\eta_k)_t| \right) \le \varepsilon_2^p
\end{equation}
and moreover for each $k$  there exist $r>0$ and a point $(\tau,\xi) \in \bR\times\bR^{d}_+$ such that $0<r < \rho \xi_1$ and $\text{supp} \, \eta_k \subset Q_r(\tau,\xi)=Q^+_r(\tau,\xi)$; we will use Proposition \ref{proposition_main_1} upon translations.
The constant $\varepsilon_2$ will be specified shortly. 

Meanwhile, observe that each $u_k : =\eta_k u $, a localization of $u$, satisfies
\begin{multline*}
-({u_k})_t + A^{ij}(t,x)D_{ij} u_k  - \lambda u_k
\\
= \eta_k f  +  A^{ij} (t,x)( D_j \eta_k D_i u+ D_i \eta_k D_j u) + D_{ij} \eta_k A^{ij} (t,x) u  - (\eta_k)_t u 
\end{multline*}
in $(-\infty,T)\times\bR^{d}_+$.
Then using a translation argument and Proposition \ref{proposition_main_1} with $\varepsilon \in (0,\varepsilon_0]$ there, we have
\begin{eqnarray*}
&&\lambda \|M u_k\|_{p,\theta}+\sqrt{\lambda}\|M Du_k\|_{p,\theta}
+ \|u_k\|_{\fH_{p,\theta}^2}\\
&\le& N_0 \|M\eta_k \,f \|_{p,\theta}
+ N_0 \sum_{i,j}\|M D_i\eta_k\, D_ju \|_{p,\theta} + N_0\|M ( D^2 \eta_k)\, u  \|_{p,\theta}\\
&&\quad + N_0 \|M  (\eta_k)_t u \|_{p,\theta} + N_1\varepsilon^{1/(\beta'q)} \|MD^2(\eta_k u)\|_{p,\theta},
\end{eqnarray*}
where $N_0=N_0(d,d_1,\delta,p,\theta)$, $N_1 = N_1(d,d_1,\delta,p,\theta,\rho)$, and $q, \beta'$ are positive numbers determined by $p$ and $\theta$.
From this and the properties of $\eta_k$ in \eqref{eq_main_4}, we obtain
\begin{eqnarray*}
&&\lambda \|M u\|_{p,\theta}+\sqrt{\lambda}\|MDu\|_{p,\theta}
+ \|u\|_{\fH_{p,\theta}^2}\\
&\le& N_0 \|M f\|_{p,\theta}
+ N_0 \varepsilon_2 \left(\|Du\|_{p,\theta} + \|M^{-1}u\|_{p,\theta}\right)\\
&&
+N_1\varepsilon^{1/(\beta'q)} \big(\|MD^2u\|_{p,\theta}
+\varepsilon_2 \|Du\|_{p,\theta} + \varepsilon_2\|M^{-1}u\|_{p,\theta}\big).
\end{eqnarray*}

Having had this, we now first choose $\varepsilon_2 \in (0,1)$ sufficiently small depending only on  $d$, $d_1$, $\delta$, $p$, and $\theta$ such that $N_0 \varepsilon_2 < 1/3$, then choose $\rho=\rho(\varepsilon_2)\in (1/2,1)$ such that \eqref{eq_main_4} is satisfied, and finally choose $\varepsilon = \varepsilon(d,d_1,\delta,p,\theta,\rho) \in (0,\varepsilon_0]$ so that
\begin{equation}
\label{eqn 5.11.1}
N_1 \varepsilon^{1/(\beta' q)} < 1/3.
\end{equation}
Upon these choices in order we arrive at the estimate \eqref{eq_main_1}.

3. General case of $B^i$s and $C$ under $\text{A}(\rho,\varepsilon)$ condition. Our system now is
\begin{equation*}			
- u_t + A^{ij}(t,x)D_{ij}u -\lambda u = f-B^iD_iu-Cu
\end{equation*}
in $(-\infty,T) \times \bR^d_+$.  Thus, by the result of step 2 and $\text{A}(\rho,\varepsilon)$ condition, if $\varepsilon\in (0,\varepsilon_0]$ satisfies (\ref{eqn 5.11.1}), we end up with
$$
\lambda \|M u\|_{p,\theta}+\sqrt{\lambda}\|MDu\|_{p,\theta}
+ \|u\|_{\fH_{p,\theta}^2}
\le N_2 \|M f\|_{p,\theta}+N_2 \varepsilon \|Du\|_{p,\theta}+N_2 \varepsilon \|M^{-1}u\|_{p,\theta},
$$
where $N_2=N_2(d,d_1,\delta,p, \theta)$. Thus it is enough to take $\varepsilon>0$ which is further smaller so that $N_2 \varepsilon<1/2$ and  the estimate \eqref{eq_main_1} holds. The taken $\varepsilon$ is still in $(0,\varepsilon_0]$.

4.  The a priori estimate \eqref{eq_main_1} holds now and hence  theorem is proved.
\end{proof}


\bibliographystyle{plain}

\begin{thebibliography}{10}

\bibitem{MR2680179}
Sun-Sig Byun and Lihe Wang.
\newblock Elliptic equations with measurable coefficients in {R}eifenberg
  domains.
\newblock {\em Adv. Math.}, 225(5):2648--2673, 2010.

\bibitem{MR4156495}
Hongjie Dong.
\newblock Recent progress in the {$L_p$} theory for elliptic and parabolic
  equations with discontinuous coefficients.
\newblock {\em Anal. Theory Appl.}, 36(2):161--199, 2020.

\bibitem{MR2601069}
Hongjie Dong and Doyoon Kim.
\newblock Elliptic equations in divergence form with partially {BMO}
  coefficients.
\newblock {\em Arch. Ration. Mech. Anal.}, 196(1):25--70, 2010.

\bibitem{MR2764911}
Hongjie Dong and Doyoon Kim.
\newblock {$L_p$} solvability of divergence type parabolic and elliptic systems
  with partially {BMO} coefficients.
\newblock {\em Calc. Var. Partial Differential Equations}, 40(3-4):357--389,
  2011.

\bibitem{MR2771670}
Hongjie Dong and Doyoon Kim.
\newblock On the {$L_p$}-solvability of higher order parabolic and elliptic
  systems with {BMO} coefficients.
\newblock {\em Arch. Ration. Mech. Anal.}, 199(3):889--941, 2011.

\bibitem{MR3318165}
Hongjie Dong and Doyoon Kim.
\newblock Elliptic and parabolic equations with measurable coefficients in
  weighted {S}obolev spaces.
\newblock {\em Adv. Math.}, 274:681--735, 2015.

\bibitem{MR3812104}
Hongjie Dong and Doyoon Kim.
\newblock On {$L_p$}-estimates for elliptic and parabolic equations with
  {$A_p$} weights.
\newblock {\em Trans. Amer. Math. Soc.}, 370(7):5081--5130, 2018.

\bibitem{MR775683}
P.~Grisvard.
\newblock {\em Elliptic problems in nonsmooth domains}, volume~24 of {\em
  Monographs and Studies in Mathematics}.
\newblock Pitman (Advanced Publishing Program), Boston, MA, 1985.

\bibitem{MR2338417}
Doyoon Kim and N.~V. Krylov.
\newblock Elliptic differential equations with coefficients measurable with
  respect to one variable and {VMO} with respect to the others.
\newblock {\em SIAM J. Math. Anal.}, 39(2):489--506, 2007.

\bibitem{MR2300337}
Doyoon Kim and N.~V. Krylov.
\newblock Parabolic equations with measurable coefficients.
\newblock {\em Potential Anal.}, 26(4):345--361, 2007.

\bibitem{MR3147235}
Ildoo Kim, Kyeong-Hun Kim, and Kijung Lee.
\newblock A weighted {$L_p$}-theory for divergence type parabolic {PDE}s with
  {BMO} coefficients on {$C^1$}-domains.
\newblock {\em J. Math. Anal. Appl.}, 412(2):589--612, 2014.

\bibitem{MR2111792}
Kyeong-Hun Kim and N.~V. Krylov.
\newblock On the {S}obolev space theory of parabolic and elliptic equations in
  {$C^1$} domains.
\newblock {\em SIAM J. Math. Anal.}, 36(2):618--642, 2004.

\bibitem{MR2990037}
Kyeong-Hun Kim and Kijung Lee.
\newblock A weighted {$L_p$}-theory for parabolic {PDE}s with {BMO}
  coefficients on {$C^1$}-domains.
\newblock {\em J. Differential Equations}, 254(2):368--407, 2013.

\bibitem{MR3470413}
Kyeong-Hun Kim and Kijung Lee.
\newblock A weighted {$L_p$}-theory for second-order parabolic and elliptic
  partial differential systems on a half space.
\newblock {\em Commun. Pure Appl. Anal.}, 15(3):761--794, 2016.

\bibitem{MR2561181}
Vladimir Kozlov and Alexander Nazarov.
\newblock The {D}irichlet problem for non-divergence parabolic equations with
  discontinuous in time coefficients.
\newblock {\em Math. Nachr.}, 282(9):1220--1241, 2009.

\bibitem{MR1262972}
N.~V. Krylov.
\newblock A {$W^n_2$}-theory of the {D}irichlet problem for {SPDE}s in general
  smooth domains.
\newblock {\em Probab. Theory Related Fields}, 98(3):389--421, 1994.

\bibitem{MR1708104}
N.~V. Krylov.
\newblock Weighted {S}obolev spaces and {L}aplace's equation and the heat
  equations in a half space.
\newblock {\em Comm. Partial Differential Equations}, 24(9-10):1611--1653,
  1999.
  
\bibitem{MR1690093}
N.~V. Krylov.
\newblock Weighted {S}obolev spaces and the heat equation in the whole space.
\newblock {\em Appl. Anal.}, 71(1-4):111--126, 1999.


\bibitem{MR1837532}
N.~V. Krylov.
\newblock Some properties of traces for stochastic and deterministic parabolic
  weighted {S}obolev spaces.
\newblock {\em J. Funct. Anal.}, 183(1):1--41, 2001.

\bibitem{MR2304157}
N.~V. Krylov.
\newblock Parabolic and elliptic equations with {VMO} coefficients.
\newblock {\em Comm. Partial Differential Equations}, 32(1-3):453--475, 2007.

\bibitem{MR1720129}
N.~V. Krylov and S.~V. Lototsky.
\newblock A {S}obolev space theory of {SPDE}s with constant coefficients in a
  half space.
\newblock {\em SIAM J. Math. Anal.}, 31(1):19--33, 1999.

\bibitem{MR802206}
Alois Kufner.
\newblock {\em Weighted {S}obolev spaces}.
\newblock A Wiley-Interscience Publication. John Wiley \& Sons Inc., New York,
  1985.
\newblock Translated from the Czech.

\bibitem{MR4072650}
Nick Lindemulder.
\newblock Maximal regularity with weights for parabolic problems with
  inhomogeneous boundary conditions.
\newblock {\em J. Evol. Equ.}, 20(1):59--108, 2020.

\bibitem{MR745140}
Jos{\'e}~L. Rubio~de Francia.
\newblock Factorization theory and {$A_{p}$} weights.
\newblock {\em Amer. J. Math.}, 106(3):533--547, 1984.

\end{thebibliography}

\def\cprime{$'$}

\end{document}